\documentclass[a4paper,12pt,leqno]{amsart}
\usepackage{color}
\usepackage{amsfonts}
\usepackage{amssymb}
\usepackage{amsmath}
\usepackage{amsthm}
\usepackage{hyperref}
\usepackage[
    backend=biber,
    style=numeric,
    sorting=nyt,
    maxbibnames=99,
    eprint=true,
    url=true,
    doi=true
]{biblatex}

\usepackage{mathtools}
\usepackage{enumitem}
\usepackage{pgfplots}
\pgfplotsset{compat=1.18}
\usepackage{xcolor}
\usepackage{mathrsfs}  
\usepackage{graphicx}
\usepackage{subcaption}
\usepackage{float}
\usepackage{fnpct}
\usepackage{bm}
\usepackage{dsfont}
\usepackage{scalerel,stackengine}
\usepackage{tikz}
\usepackage{tikz-3dplot}
\usepackage{caption}
\usepackage[margin=1in]{geometry}
\stackMath
\newcommand\reallywidehat[1]{%
\savestack{\tmpbox}{\stretchto{%
  \scaleto{%
    \scalerel*[\widthof{\ensuremath{#1}}]{\kern-.6pt\bigwedge\kern-.6pt}%
    {\rule[-\textheight/2]{1ex}{\textheight}}
  }{\textheight}%
}{0.5ex}}%
\stackon[1pt]{#1}{\tmpbox}%
}
\usepackage{bbm}

\makeatletter
\DeclareRobustCommand\widecheck[1]{{\mathpalette\@widecheck{#1}}}
\def\@widecheck#1#2{%
    \setbox\z@\hbox{\m@th$#1#2$}%
    \setbox\tw@\hbox{\m@th$#1%
       \widehat{%
          \vrule\@width\z@\@height\ht\z@
          \vrule\@height\z@\@width\wd\z@}$}%
    \dp\tw@-\ht\z@
    \@tempdima\ht\z@ \advance\@tempdima2\ht\tw@ \divide\@tempdima\thr@@
    \setbox\tw@\hbox{%
       \raise\@tempdima\hbox{\scalebox{1}[-1]{\lower\@tempdima\box
\tw@}}}%
    {\ooalign{\box\tw@ \cr \box\z@}}}
\makeatother

\newcommand{\Comment}[1]{}

\def\C{\mathbb{C}}
\def\R{\mathbb{R}}

\def\Z{\mathbb{Z}}
\def\N{\mathbb{N}}

\def\F{\mathbb{F}}

\newtheorem{thm}{Theorem}[section]
\newtheorem{cor}[thm]{Corollary}
\newtheorem{lemma}[thm]{Lemma}
\newtheorem{prop}[thm]{Proposition}
\theoremstyle{definition}
\newtheorem{remark}[thm]{Remark}
\newtheorem{defn}[thm]{Definition}
\newtheorem{pr}{Problem}[section]
\newtheorem{qt}{Question}[section]

\newtheorem{conj}{Conjecture}
\newtheorem{eg}[thm]{Example}
\usepackage{enumitem}
\newlist{casesp}{enumerate}{3} 
\setlist[casesp]{align=left, 
                 listparindent=\parindent, 
                 parsep=\parskip, 
                 font=\normalfont\bfseries, 
                 leftmargin=0pt, 
                 labelwidth=0pt, 
                 itemindent=.4em,labelsep=.4em, 
                 partopsep=0pt, 
                 }
\setlist[casesp,1]{label=Case~\arabic*:,ref=\arabic*}
\setlist[casesp,2]{label=Case~\thecasespi.\roman*:,ref=\thecasespi.\roman*}
\setlist[casesp,3]{label=Case~\thecasespii.\alph*:,ref=\thecasespii.\alph*}

\numberwithin{equation}{section}

\def\R{\mathbb R}

\def\beas{\begin{eqnarray*}}
\def\eeas{\end{eqnarray*}}

\newcounter{vremennyj}

\begin{document}

\author{Eyvindur Ari Palsson}
\address[E.~ Palsson]{Department of Mathematics, Virginia Tech, Blacksburg, VA 24061}
\email{palsson@vt.edu}
\author{Jian-An Wang}
\address[J.-A.~ Wang]{Department of Mathematics, Virginia Tech, Blacksburg, VA 24061}
\email{jiananwang@vt.edu}

\title[]{On Densities in the Erd\H{o}s-Falconer Distance Problems over the Integers}
\maketitle

\begin{abstract}
Motivated by the classical Erd\H{o}s and Falconer distance problems and Magyar’s work on distance sets in \(\mathbb Z^d\), we study discrete analogues of the Erd\H{o}s-Falconer unpinned and pinned distance problems over the integers. We ask how the discrete Hausdorff dimension of a set \(A\subseteq\mathbb Z^d\) influences the size of its squared-distance sets, measured by generalized upper natural density and upper discrete mass dimension. We propose conjectures for both problems and prove the unpinned conjecture in dimensions \(d=1,2,3\) and the pinned conjecture in dimensions \(d=1,2\). For the unpinned problem in higher dimensions, we obtain explicit lower bounds for the dimensions of distance sets using elementary integer sphere counting, results on the Erd\H{o}s distinct distances problem, and \(\ell^p\)-improving inequalities for discrete spherical averages. For the pinned problem, we obtain lower bounds for the dimensions of distance sets for some or every pin. We also establish more general results involving upper discrete mass dimension, upper discrete Banach dimension, and affine dimension, including bounds on the affine dimension of the exceptional set of pins. We construct examples showing that both conjectures are sharp in low dimensions. Further examples demonstrate how the distinctions among discrete Hausdorff dimension, upper discrete mass dimension, and upper discrete Banach dimension lead to different phenomena in the unpinned and pinned distance problems, including limitations on the nontrivial upper bounds these dimensions can provide for the dimensions of distance sets.
\end{abstract}

\section{Introduction}\label{Intro}
Let $R$ be a ring, and let $A$ be a subset of $R^d$. We define the \textbf{distance set} $\Delta^i(A)$ of $A$ for $i=1,2$ by
$$
\left\{
\begin{array}{ll}
     \Delta^1(A)=\left\{\|x-y\|\in R:x,y\in A\right\} &\text{if }R=\R  \vspace{0.1cm}\\
    \Delta^2(A)=\left\{\|x-y\|^2\in R:x,y\in A\right\} &\text{if }R=\R,\Z,\F_q 
\end{array}
\right.
$$
and the \textbf{pinned distance set} $\Delta^i(A,x)$ of $A$ and $x\in R^d$ by
$$
\left\{
\begin{array}{ll}
     \Delta^1(A,x)=\left\{\|x-y\|\in R:y\in A\right\} &\text{if }R=\R  \vspace{0.1cm}\\
    \Delta^2(A,x)=\left\{\|x-y\|^2\in R:y\in A\right\} &\text{if }R=\R,\Z,\F_q 
\end{array}
\right.
$$
where
$$
\|x-y\|^2=(x_1-y_1)^2+\cdots+(x_d-y_d)^2
$$
We write $x\lesssim y$ if there exists a constant $C>0$ such that $x\leq C y$. If the constant $C$ depends on a parameter $s$ but not $x,y$, we write it as $x\lesssim_{s} y$. Moreover, we write $x\approx y$ if $x\lesssim y$ and $x\gtrsim y$. Last but not least, we write $x\lessapprox y$ with the controlling parameter $n$ if for every $\epsilon>0$ there exists $C_{\epsilon}>0$ such that $x\leq C_{\epsilon}n^{\epsilon}y$.
\subsection{Background}
\subsubsection{Distinct Distance Problems}
In 1946, the distinct distance problem was first studied by Erd\H{o}s \cite{Erdos01051946}, who posed the following problem:
\begin{pr}
What is the least number of distinct distances determined by $n$ points in $\R^d$?
\end{pr}
Erd\H{o}s considered a lattice containing $n$ points and proposed the following conjecture:
\begin{conj}[Erd\H{o}s]
Let $d\geq2$ and $A\subseteq\R^d$ with $|A|=n$. Then
$$
\left|\Delta^2(A)\right|\gtrsim
\left\{
\begin{array}{ll}
    \frac{n}{\sqrt{\log{n}}}&\text{if }d=2  \vspace{0.1cm}\\
    n^{\frac{2}{d}}&\text{if }d\geq3
\end{array}
\right.
$$
\end{conj}
The Erd\H{o}s distance problem, which plays an important role in discrete geometry and arithmetic combinatorics, has been studied extensively; see \cite{sheffer2014distinct} for more on its history and variants. There are numerous famous techniques and theorems involved in the study of the Erd\H{o}s distance problem such as the Szemer\'edi–Trotter theorem from incidence geometry and polynomial partitioning from algebraic geometry; see for example \cite{MR4394303,szemeredi1983extremal}. The state-of-the-art results on the Erd\H{o}s distance problem are proved by Guth and Katz \cite{guth2015erdHos} for $d=2$, by Tidor, Yu, and Zakharov \cite{tidor2026erdhosdistinctdistancesproblem} for $d=3$, and by Solymosi and Vu \cite{solymosi2008near} for $d\geq4$. It is noteworthy that Guth and Katz derived the bound $\gtrsim\frac{n}{\log{n}}$, so they actually proved that the Erd\H{o}s conjecture is true for $d=2$ up to polylogarithmic factors. The recent breakthrough by Tidor, Yu, and Zakharov shows the bound $\gtrapprox n^{\frac{2}{3}}$, so the Erd\H{o}s conjecture is also true for $d=3$ up to subpolynomial factors.

In the 1980s, Falconer \cite{Falconer_1985} posed a famous distance problem, which is regarded as a continuous analog of the Erd\H{o}s distance problem, in geometric measure theory:
\begin{pr}
How large does the Hausdorff dimension of a compact subset $A\subseteq\R^d$ need to be to ensure that the distance set $\Delta^1(A)$ has positive Lebesgue measure?
\end{pr}
Falconer constructed an example inspired by the lattice construction for the Erd\H{o}s distance problem and proposed the following conjecture:
\begin{conj}[Falconer]
    Let $d\geq2$, and let $A$ be a compact subset of $\R^d$. Then
    $$
    \dim_{\mathcal{H}}(A)>\frac{d}{2}\Longrightarrow\mathcal{L}_1\left(\Delta^1(A)\right)>0
    $$
    where $\dim_{\mathcal{H}}$ is the Hausdorff dimension and $\mathcal{L}_1$ is the Lebesgue measure.
\end{conj}
The Falconer conjecture, which remains open in all dimensions and plays an important role in harmonic analysis, has attracted much attention in recent years. Recent progress \cite{du2019sharp,guth2020falconer,du2021weighted,du2021improved,du2023weighted} towards the Falconer conjecture shows that the current best thresholds are $\frac{5}{4}$ for $d=2$ and $\frac{d}{2}+\frac{1}{4}-\frac{1}{8d+4}$ for $d\geq3$.

The Erd\H{o}s and Falconer distance problems also motivate researchers to dive into variants of them, for instance, Iosevich and Rudnev {\cite{iosevich2007erdos}} stated a finite field $\F_q^d$-variant of the Erd\H{o}s-Falconer conjecture and related it to number theory, see also \cite{hart2011averages,koh2022note}.
\begin{conj}[Iosevich and Rudnev]\label{finitefields}
    Let $\F_q$ be a finite field with odd $q$ elements. Suppose that $d\geq2$ is even or $d,q\equiv3\pmod{4}$. Let $A$ be a subset of $\F_q^d$. Then
    $$
    |A|\gtrsim|\F_q|^{\frac{d}{2}}\Longrightarrow\left|\Delta^2(A)\right|\gtrsim|\F_q|
    $$
\end{conj}
\subsubsection{Pinned Distance Problems}
Erd\H{o}s \cite{MR411984} also formulated the pinned distance problem and conjecture: 
\begin{pr}
Does every finite set $A\subseteq\R^d$ contain a point $x$ such that the behavior of $\left|\Delta^2(A)\right|$ and $\left|\Delta^2(A,x)\right|$ is asymptotically the same?
\end{pr}
\begin{conj}
Let $d\geq2$ and $A\subseteq\R^d$ with $|A|=n$. Then there exists $x\in A$ such that
$$
\left|\Delta^2(A,x)\right|\gtrsim
\left\{
\begin{array}{ll}
    \frac{n}{\sqrt{\log{n}}}&\text{if }d=2  \vspace{0.1cm}\\
    n^{\frac{2}{d}}&\text{if }d\geq3
\end{array}
\right.
$$
\end{conj}
It is clear that $\left|\Delta^2(A)\right|\geq\left|\Delta^2(A,x)\right|$, so the pinned distance variant is stronger than the original distance problem. So far the current best results on the pinned distance problem are $\max_{x\in A}\left|\Delta^2(A,x)\right|\gtrsim|A|^{0.8641\dots}$ proved by Katz and Tardos \cite{MR2065258} for $d=2$ and applying Solymosi and Vu's recursion \cite{solymosi2008near} to Katz and Tardos's bound to derive the strongest result for $d\geq3$.

The pinned variant of the Falconer distance problem and its associated conjecture can similarly be formulated as follows (See for example \cite{MR4226184}):
\begin{pr}
How large does the Hausdorff dimension of a compact subset $A\subseteq\R^d$ need to be to guarantee that there exists a point $x\in A$ such that the pinned distance set $\Delta^1(A,x)$ has positive Lebesgue measure?
\end{pr}
\begin{conj}
Let $d\geq2$, and let $A$ be a compact subset of $\R^d$. Then there exists $x\in A$ such that
$$
\dim_{\mathcal{H}}(A)>\frac{d}{2}\Longrightarrow\mathcal{L}_1\left(\Delta^1(A,x)\right)>0
$$
\end{conj}
The Falconer pinned distance problem was first studied by Peres and Schlag \cite{MR1749437}. Liu \cite{MR3925111} made a connection between the main proof techniques for the pinned and unpinned variants; hence, all the best results on the pinned and unpinned variants have been the same.

The Erd\H{o}s and Falconer pinned distance conjectures also motivate researchers to study variants of them, for example, the finite field variant \cite{MR4411329,pham2026erdosfalconerdistanceproblemdimensions,pham2026sharpsphericalextensiontheorem,koh2022note} of the Erd\H{o}s and Falconer pinned distance conjectures:
\begin{conj}
Let $\F_q$ be a finite field with odd $q$ elements. Suppose that $d\geq2$ is even or $d,q\equiv3\pmod{4}$. Let $A$ be a subset of $\F_q^d$. Then there exists $x\in A$ such that
$$
|A|\gtrsim|\F_q|^{\frac{d}{2}}\Longrightarrow\left|\Delta^2(A,x)\right|\gtrsim|\F_q|
$$
\end{conj}
\subsubsection{Densities and Distance Sets over Integers}
In order to study the distance problems over integers, the concept of measuring subsets of $\Z^d$ is essential. One of the ideas is to consider densities in light of the similarity between densities and measures such as subadditivity and translation invariance; see \cite{buck1946measure,leonetti2020notions} for more on their similarity.

In 1990, Furstenberg, Katznelson, and Weiss \cite{article} established the following result involving distance sets and densities in $\R^2$: If $A$ is a measurable subset of $\R^2$, then
$$
\delta^{*}(A)>0\Longrightarrow [\lambda_A,\infty)\subseteq\Delta^1(A)\text{ for some }\lambda_A\in\R
$$
where
$$
\delta^{*}(A)=\lim_{N\to\infty}\sup_{t\in\R^d}\frac{\mathcal{L}_d\big(A\cap(t+C_N)\big)}{\mathcal{L}_d(C_N)}
$$
is the \textbf{upper Banach density} of $A\subseteq\R^d$, here $C_N$ denotes the cube $\left[-\frac{N}{2},\frac{N}{2}\right]^d$.

Based on Fourier analytic methods, Magyar \cite{magyar2008distance} proved a discrete analog of Furstenberg, Katznelson, and Weiss's result: If $d\geq5$ and $A\subseteq\Z^d$, then
$$
\delta^{*}(A)>0\Longrightarrow q_{d,\delta^{*}(A)}^2\Z\cap[\lambda_A,\infty)\subseteq\Delta^2(A)\text{ for some }\lambda_A\in\R\text{ and }q_{d,\delta^{*}(A)}\in\Z
$$
where
$$
\delta^{*}(A)=\lim_{N\to\infty}\sup_{t\in\Z^d}\frac{|A\cap(t+Q_N)|}{|Q_N|}
$$
is the \textbf{upper Banach density} of $A\subseteq\Z^d$, here $Q_N$ denotes the discrete cube $C_N\cap\Z^d$.

{Magyar's discrete analog implies the following statement: If $d\geq5$ and $A\subseteq\Z^d$, then
$$
\delta^{*}(A)>0\Longrightarrow \mathcal{N}^{*}\big(\Delta^2(A)\big)\geq\frac{1}{2q_{d,\delta^{*}(A)}^2}>0
$$
where
$$
\mathcal{N}^{*}\big(\Delta^2(A)\big)=\limsup_{N\to\infty}\frac{\left|\Delta^2(A)\cap\left[-\frac{N}{2},\frac{N}{2}\right]\right|}{\left|\left[-\frac{N}{2},\frac{N}{2}\right]\cap\Z\right|}
$$
is the \textbf{upper natural density} of $\Delta^2(A)\subseteq\Z$. Inspired by this statement, we are curious whether the assumption $\delta^{*}(A)>0$ can be weakened to get a weaker conclusion $\mathcal{N}_{\theta}^{*}\big(\Delta^2(A)\big)>0$ for a fixed $\theta\in[0,1]$, where
$$
\mathcal{N}_{\theta}^{*}\big(\Delta^2(A)\big)=\limsup_{N\to\infty}\frac{\left|\Delta^2(A)\cap\left[-\frac{N}{2},\frac{N}{2}\right]\right|}{\left|\left[-\frac{N}{2},\frac{N}{2}\right]\cap\Z\right|^{\theta}}
$$
is the $\bm{\theta}$\textbf{-upper natural density} of $\Delta^2(A)\subseteq\Z$. In other words, we are interested in what kinds of sparser subsets $A\subseteq\Z^d$ can have $\mathcal{N}_{\theta}^{*}\big(\Delta^2(A)\big)>0$ for a fixed $\theta\in[0,1]$. Thus, the language of discrete fractal dimensions is involved to study sparser subsets of $\Z^d$ and makes our problems resemble the Falconer distance problems.
}
\subsection{Main Results}
Inspired by the Erd\H{o}s and Falconer distance problems, their variants, and Magyar's discrete analog of Furstenberg, Katznelson, and Weiss's result, we set up the following discrete $\Z^d$-variant of the Erd\H{o}s-Falconer distance problem:
\begin{pr}\label{unp1}
    Fix $\theta\in[0,1]$. How large does the discrete Hausdorff dimension $\dim_{\mathcal{H}}^{\text{dis}}(A)$ of a subset $A$ of $\Z^d$ {need to be to ensure that} the distance set $\Delta^2(A)$ has positive $\theta$-upper natural density?
\end{pr}
Here, the $\theta$-upper natural density $\mathcal{N}_{\theta}^{*}\big(\Delta^2(A)\big)$ is a generalization of the usual upper natural density $\mathcal{N}^{*}\big(\Delta^2(A)\big)$ and will be viewed as a discrete analog of the Lebesgue measure in the Falconer conjecture.

Moreover, there is an equivalent problem:
\begin{pr}\label{unp2}
Can we find a lower bound\footnote{No nontrivial upper bound for $\dim_{\mathcal{UM}}^{\mathrm{dis}}\big(\Delta^2(A)\big)$ can depend solely on $\dim_{\mathcal H}^{\mathrm{dis}}(A)$ or $\dim_{\mathcal{UM}}^{\text{dis}}(A)$; see Example \ref{egggg5555}.} dependent on the discrete Hausdorff dimension $\dim_{\mathcal{H}}^{\text{dis}}(A)$ of a subset $A$ of $\Z^d$ for the upper discrete mass dimension of $\Delta^2(A)$?
\end{pr}
Based on the Erd\H{o}s conjecture\footnote{We will show that the Erd\H{o}s conjecture implies our Conjecture \ref{conj} in Section \ref{D}.} and our sharp example (Example \ref{sharp}) in low dimensions by mixing Erd\H{o}s's and Falconer's sharp examples, we propose the following conjecture:
\begin{conj}\label{conj}
    Let $A\subseteq\Z^d$. Then\footnote{Note that Conjecture \ref{conj} is applicable only when $\theta\in\left[0,\frac{1}{2}\right)$ if $d=1$ and $\theta\in[0,1)$ if $d\geq2$. This is similar to the fact that the Falconer conjecture fails trivially for $d=1$.}
    $$
    \dim_{\mathcal{H}}^{\text{dis}}(A)>
    \left\{
    \begin{array}{ll}
          2\theta &\text{if }d=1\\
         d\theta &\text{if }d\geq2 
    \end{array}
    \right.
    \Longrightarrow \mathcal{N}_{\theta}^{*}\big(\Delta^2(A)\big)>0
    $$
    or equivalently,
    $$
    \dim_{\mathcal{UM}}^{\text{dis}}\big(\Delta^2(A)\big)\geq
    \left\{
    \begin{array}{ll}
        \frac{\dim_{\mathcal{H}}^{\text{dis}}(A)}{2} &\text{if }d=1 \vspace{0.1cm}\\
        \frac{\dim_{\mathcal{H}}^{\text{dis}}(A)}{d} &\text{if }d\geq2 
    \end{array}
    \right.
    $$
\end{conj}
One can see this conjecture is similar to the Falconer conjecture if $\theta=\frac{1}{2}$. We can interpret $\mathcal{N}_{\frac{1}{2}}^{*}\big(\Delta^2(A)\big)>0$ as the fact that $\Delta^2(A)$ has the same (or greater) proportion to the size of the set of perfect squares.

Define the elementary sphere bound
$$
T_{\mathrm{Sph}}(d,\theta)=\left\{
\begin{array}{ll}
    2\theta &\text{if }d=1,2  \\
    d-2+2\theta &\text{if }d\geq3
\end{array}
\right.
$$
and the finite Erd\H{o}s distance bound 
$$
T_{\mathrm{Erd}}(d,\theta)=\left\{
\begin{array}{ll}
    2\theta &\text{if }d=1,2  \\
    \frac{d^2+2d+c_d}{d+1}\theta &\text{if }d\geq3
\end{array}
\right.
\text{ where }c_d=
\left\{
\begin{array}{ll}
    -3+\frac{4}{d^2}&\text{if }d\text{ is even}\\
    -3&\text{if }d\text{ is odd}
\end{array}
\right.
$$
For $d\geq5$, define the $\ell^p$-improving bound
$$
T_{\ell^p}(d,\theta)=\frac{(d+2)\theta+(d-2)}{2}
$$
Our first main theorem can partially solve Problem \ref{unp1} and Problem \ref{unp2}:
\begin{thm}\label{mainthm2}
Let $A\subseteq\Z^d$ and $0\leq\theta\leq1$. Then
$$
\dim_{\mathcal{H}}^{\text{dis}}(A)>T_{d}(\theta)
\Longrightarrow \mathcal{N}_{\theta}^{*}\big(\Delta^2(A)\big)>0
$$
where
$$
T_{d}(\theta)=\left\{
\begin{array}{ll}
    T_{\mathrm{Sph}}(d,\theta)=T_{\mathrm{Erd}}(d,\theta) &\text{if }d=1,2  \\
    T_{\mathrm{Erd}}(d,\theta) &\text{if }d=3\\
    \min\left\{T_{\mathrm{Erd}}(d,\theta),T_{\mathrm{Sph}}(d,\theta)\right\}&\text{if }d=4\\
    \min\left\{T_{\mathrm{Erd}}(d,\theta),T_{\ell^p}(d,\theta)\right\}&\text{if }d\geq5
\end{array}
\right.
$$
in other words,
$$
T_{d}(\theta)=\left\{
\begin{array}{ll}
    2\theta &\text{if }d=1,2  \\
    3\theta &\text{if }d=3\\
    \min\left\{\frac{17}{4}\theta,2+2\theta\right\}&\text{if }d=4\\
    \min\left\{\frac{d^2+2d+c_d}{d+1}\theta,\frac{(d+2)\theta+(d-2)}{2}\right\}&\text{if }d\geq5
\end{array}
\right.
$$
or equivalently,
$$
\dim_{\mathcal{UM}}^{\text{dis}}\big(\Delta^2(A)\big)\geq
\left\{
\begin{array}{ll}
\frac{\dim_{\mathcal{H}}^{\text{dis}}(A)}{2}&\text{if }d=1,2\\
\frac{\dim_{\mathcal{H}}^{\text{dis}}(A)}{3}&\text{if }d=3\\
\max\left\{\frac{4\dim_{\mathcal{H}}^{\text{dis}}(A)}{17},\frac{\dim_{\mathcal{H}}^{\text{dis}}(A)-2}{2}\right\}&\text{if }d=4\\
\max\left\{\frac{\dim_{\mathcal{H}}^{\text{dis}}(A)(d+1)}{d^2+2d+c_d},\frac{2\dim_{\mathcal{H}}^{\text{dis}}(A)-d+2}{d+2}\right\}&\text{if }d\geq5
\end{array}
\right.
$$
In particular, Conjecture \ref{conj} is true for $d=1,2,3$.
\end{thm}
The progress towards the Erd\H{o}s conjecture is highly relevant to our Conjecture \ref{conj}. Actually, the finite Erd\H{o}s distance bound\footnote{Unlike in the unpinned case, the results on the Erd\H{o}s pinned distance problem cannot be applied to our discrete $\Z^d$-variant of the Erd\H{o}s-Falconer pinned distance problem. We will explain this in Section \ref{D}.} is based on the current best result on the Erd\H{o}s distance problem, and the Erd\H{o}s conjecture implies our Conjecture \ref{conj}. In particular, Guth and Katz's result can derive the finite Erd\H{o}s distance bound and imply our Conjecture \ref{conj} for $d=1,2$; nonetheless, our proof based on the elementary sphere bound is much shorter and easier to follow. This illustrates the significant difference between the Erd\H{o}s distance problem and our discrete $\Z^d$-variant of the Erd\H{o}s-Falconer distance problem.

Here, we make a comparison\footnote{In Section \ref{M}, we will compare the elementary sphere bound with the finite Erd\H{o}s distance bound.} between the $\ell^p$-improving bound and the finite Erd\H{o}s distance bound in Theorem \ref{mainthm2}: one can see that the $\ell^p$-improving bound is better than the finite Erd\H{o}s distance bound when $\theta\geq\theta_{\text{stronger}}^{(\ell^p,\mathrm{Erd})}$ or equivalently $\dim_{\mathcal{H}}^{\text{dis}}(A)\geq d_{\text{stronger}}^{(\ell^p,\mathrm{Erd})}$, where
$$
\theta_{\text{stronger}}^{(\ell^p,\mathrm{Erd})}=\frac{(d-2)(d+1)}{d^2+d+2c_d-2}
\quad\quad\quad
d_{\text{stronger}}^{(\ell^p,\mathrm{Erd})}=\frac{(d-2)(d^2+2d+c_d)}{d^2+d+2c_d-2}
$$
Additionally, the finite Erd\H{o}s distance bound is inapplicable when $\theta\geq\theta_{\text{inapplicable}}^{\text{Erd}}$, where
$$
\theta_{\text{inapplicable}}^{\text{Erd}}=\frac{d(d+1)}{d^2+2d+c_d}
$$
but in the meantime, the $\ell^p$-improving bound is applicable when $\theta\in[0,1)$ and hence can fill this gap.
\def\dVal{5}


\pgfmathsetmacro{\cVal}{
    -3 + 4*(1-mod(\dVal,2))/(\dVal*\dVal)
}

\pgfmathsetmacro{\CErd}{
    (\dVal*\dVal + 2*\dVal + \cVal)/(\dVal + 1)
}

\pgfmathsetmacro{\lpSlope}{
    (\dVal + 2)/2
}
\pgfmathsetmacro{\lpIntercept}{
    (\dVal - 2)/2
}

\pgfmathsetmacro{\thetaStronger}{
    (\dVal - 2)*(\dVal + 1)/
    (\dVal*\dVal + \dVal + 2*\cVal - 2)
}

\pgfmathsetmacro{\dStronger}{
    \CErd*\thetaStronger
}

\pgfmathsetmacro{\thetaInapp}{
    \dVal/\CErd
}

\pgfmathsetmacro{\yCross}{
    \CErd*\thetaStronger
}

\pgfmathsetmacro{\yCrossTwo}{
    \dStronger/\CErd
}

\pgfmathsetmacro{\yTopOne}{
    max(\dVal+0.7,1.03*\CErd+0.2)
}
\pgfmathsetmacro{\xTopTwo}{
    \dVal + 0.25
}
\pgfmathsetmacro{\yTopTwo}{
    1.05
}

\pgfplotsset{
    comparison/.style={
        width=12cm,
        height=8cm,
        axis lines=left,
        axis line style={black},
        axis on top,
        clip=true,
        tick label style={font=\small},
        label style={font=\small},
        legend style={
            at={(0.05,0.95)},
            anchor=north west,
            draw=black,
            fill=white,
            fill opacity=1,
            text opacity=1,
            font=\small,
            legend columns=1
        },
        legend cell align=left,
        enlarge x limits=false
    }
}

\tikzset{
    lpcurve/.style={
        blue!85!black,
        thick,
    },
    erdoscurve/.style={
        red!85!black,
        thick,
    },
    conjecturecurve/.style={
        green!50!black,
        thick,
    },
    crossover/.style={
        gray,
    },
    inapp/.style={
        gray,
        thick
    },
    dimcap/.style={
        black,
        thin
    }
}
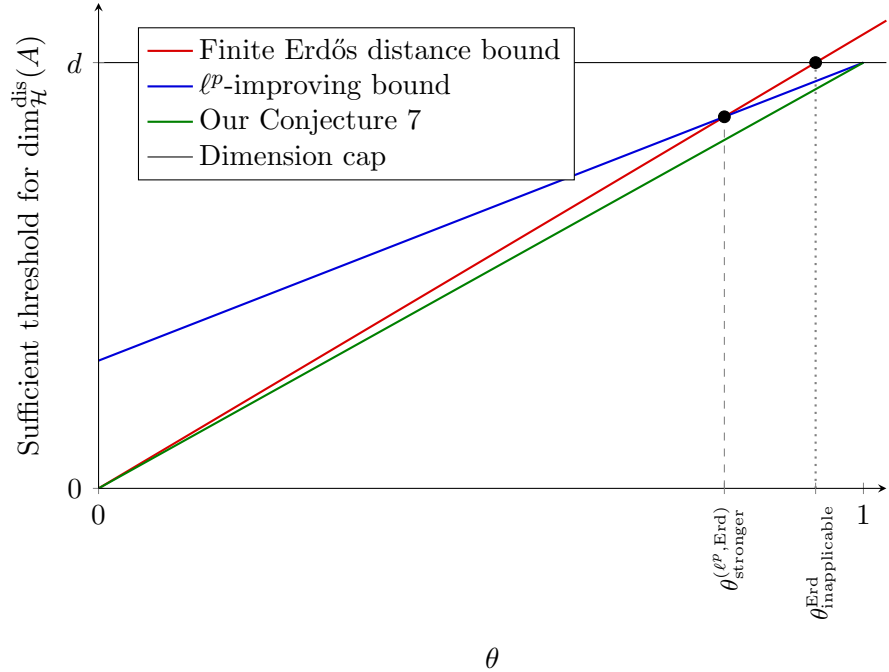
\begin{figure}[H]
\centering
\begin{tikzpicture}
\begin{axis}[
    comparison,
    xmin=0,
    xmax=1.03,
    ymin=0,
    ymax=\yTopOne,
    xlabel={$\theta$},
    ylabel={Sufficient threshold for $\dim_{\mathcal{H}}^{\mathrm{dis}}(A)$},
    xtick={0,1},
    xticklabels={$0$,$1$},
    ytick={0,\dVal},
    yticklabels={$0$,$d$},
    extra x ticks={\thetaStronger,\thetaInapp},
    extra x tick labels={
        $\theta_{\text{stronger}}^{(\ell^p,\mathrm{Erd})}$,
        $\theta_{\text{inapplicable}}^{\text{Erd}}$
    },
    extra x tick style={
        tick label style={
            rotate=90,
            anchor=east,
            font=\scriptsize,
            yshift=-2pt
        }
    }
]

\addplot[
    erdoscurve,
    domain=0:1.03,
    samples=2
] {\CErd*x};
\addlegendentry{Finite Erd\H{o}s distance bound}

\addplot[
    lpcurve,
    domain=0:1,
    samples=2
] {\lpSlope*x+\lpIntercept};
\addlegendentry{$\ell^p$-improving bound}

\addplot[
    conjecturecurve,
    domain=0:1,
    samples=2
] {\dVal*x};
\addlegendentry{Our Conjecture \ref{conj}}

\addplot[
    dimcap,
    domain=0:1.03,
    samples=2
] {\dVal};
\addlegendentry{Dimension cap}

\addplot[
    crossover,
    dashed,
    forget plot
] coordinates {
    (\thetaStronger,0)
    (\thetaStronger,\yCross)
};
\addplot[
    inapp,
    dotted,
    forget plot
] coordinates {
    (\thetaInapp,0)
    (\thetaInapp,\dVal)
};

\addplot[
    only marks,
    mark=*,
    mark size=2.2pt,
    black,
    forget plot
] coordinates {
    (\thetaStronger,\yCross)
    (\thetaInapp,\dVal)
};

\end{axis}
\end{tikzpicture}
\caption{
Comparison of the finite Erd\H{o}s distance bound, the $\ell^p$-improving bound,
our Conjecture \ref{conj} for $d=5$.
}
\label{fig:new-figure1}
\end{figure}
\begin{figure}[H]
\centering
\begin{tikzpicture}
\begin{axis}[
    comparison,
    xmin=0,
    xmax=\xTopTwo,
    ymin=0,
    ymax=\yTopTwo,
    xlabel={$\dim_{\mathcal{H}}^{\mathrm{dis}}(A)$},
    ylabel={Lower bounds for $\dim_{\mathcal{UM}}^{\mathrm{dis}}(\Delta^2(A))$},
    xtick={0,\dVal},
    xticklabels={$0$,$d$},
    ytick={0,1},
    yticklabels={$0$,$1$},
    extra x ticks={\dStronger},
    extra x tick labels={$d_{\text{stronger}}^{(\ell^p,\mathrm{Erd})}$},
    extra x tick style={
        tick label style={
            rotate=90,
            anchor=east,
            font=\scriptsize,
            yshift=-2pt
        }
    }
]

\addplot[
    erdoscurve,
    domain=0:\dVal,
    samples=2
] {x/\CErd};
\addlegendentry{Finite Erd\H{o}s distance bound}

\addplot[
    lpcurve,
    domain=0:\lpIntercept,
    samples=2
] {0};
\addlegendentry{$\ell^p$-improving bound}

\addplot[
    lpcurve,
    forget plot,
    domain=\lpIntercept:\dVal,
    samples=2
] {(2*x-\dVal+2)/(\dVal+2)};

\addplot[
    conjecturecurve,
    domain=0:\dVal,
    samples=2
] {x/\dVal};
\addlegendentry{Our Conjecture \ref{conj}}

\addplot[
    dimcap,
    domain=0:\dVal,
    samples=2
] {1};
\addlegendentry{Dimension cap}

\addplot[
    dimcap,
    forget plot
] coordinates {
    (\dVal,0)
    (\dVal,1)
};

\addplot[
    crossover,
    dashed,
    forget plot
] coordinates {
    (\dStronger,0)
    (\dStronger,\yCrossTwo)
};

\addplot[
    only marks,
    mark=*,
    mark size=2.2pt,
    black,
    forget plot
] coordinates {
    (\dStronger,\yCrossTwo)
};

\end{axis}
\end{tikzpicture}
\caption{
Comparison of the lower bounds for
$\dim_{\mathcal{UM}}^{\mathrm{dis}}(\Delta^2(A))$
obtained from the finite Erd\H{o}s distance bound,
the $\ell^p$-improving bound, and our Conjecture \ref{conj} for $d=5$.
}
\label{fig:new-figure2}
\end{figure}
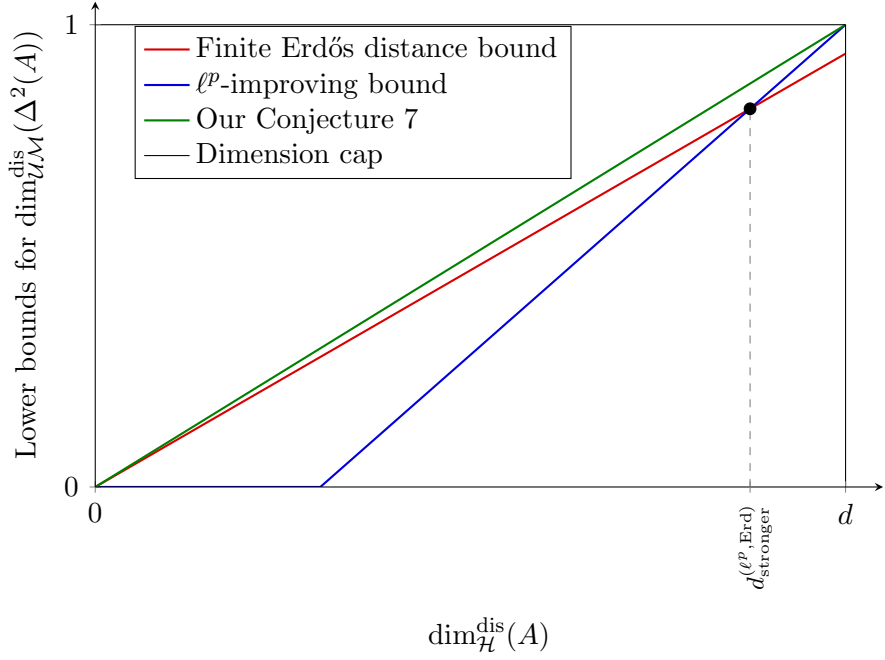
Similarly, inspired by the Erd\H{o}s and Falconer pinned distance problems, their variants, and Magyar's discrete analog of Furstenberg, Katznelson, and Weiss's result, we also pose the following discrete $\Z^d$-variant of the Erd\H{o}s-Falconer pinned distance problem:
\begin{pr}\label{pp1}
    Fix $\theta\in[0,1]$. How large does the discrete Hausdorff dimension $\dim_{\mathcal{H}}^{\text{dis}}(A)$ of a subset $A$ of $\Z^d$ {need to be to ensure that} there exists $x\in A$ such that the pinned distance set $\Delta^2(A,x)$ has positive $\theta$-upper natural density?
\end{pr}
Likewise, there is also an equivalent problem:
\begin{pr}\label{pp2}
Can we find a lower bound\footnote{No nontrivial upper bound for $\dim_{\mathcal{UM}}^{\mathrm{dis}}\big(\Delta^2(A,x)\big)$ can depend solely on $\dim_{\mathcal H}^{\mathrm{dis}}(A)$, even if one is allowed to choose the pin $x\in A$; see Example \ref{manypin}.} dependent on the discrete Hausdorff dimension $\dim_{\mathcal{H}}^{\text{dis}}(A)$ of a subset $A$ of $\Z^d$ for the upper discrete mass dimension of $\Delta^2(A,x)$ for some $x\in A$?
\end{pr}
Again, based on our sharp example (Example \ref{sharp}) in low dimensions by mixing Erd\H{o}s's and Falconer's sharp examples, we also propose the following pinned conjecture:
\begin{conj}\label{conjpin}
    Let $A\subseteq\Z^d$. Then there exists $x\in A$ such that\footnote{Note that Conjecture \ref{conjpin} is applicable only when $\theta\in\left[0,\frac{1}{2}\right)$ if $d=1$ and $\theta\in[0,1)$ if $d\geq2$. This is similar to the fact that the Falconer pinned conjecture fails trivially for $d=1$.}
    $$
    \dim_{\mathcal{H}}^{\text{dis}}(A)>
    \left\{
    \begin{array}{ll}
          2\theta &\text{if }d=1\\
         d\theta &\text{if }d\geq2 
    \end{array}
    \right.
    \Longrightarrow \mathcal{N}_{\theta}^{*}\big(\Delta^2(A,x)\big)>0
    $$
    or equivalently,
    $$
    \dim_{\mathcal{UM}}^{\text{dis}}\big(\Delta^2(A,x)\big)\geq
    \left\{
    \begin{array}{ll}
        \frac{\dim_{\mathcal{H}}^{\text{dis}}(A)}{2} &\text{if }d=1 \vspace{0.1cm}\\
        \frac{\dim_{\mathcal{H}}^{\text{dis}}(A)}{d} &\text{if }d\geq2 
    \end{array}
    \right.
    $$
\end{conj}
One can see this pinned conjecture is similar to the Falconer pinned conjecture if $\theta=\frac{1}{2}$. We can interpret $\mathcal{N}_{\frac{1}{2}}^{*}\big(\Delta^2(A,x)\big)>0$ as the fact that $\Delta^2(A,x)$ has the same (or greater) proportion to the size of the set of perfect squares.

Our second main theorem can partially solve Problem \ref{pp1} and Problem \ref{pp2}:
\begin{thm}\label{pinmainthm2}
Let $A\subseteq\Z^d$ and $0\leq\theta\leq1$. If $d=1$, then
$$
\dim_{\mathcal{H}}^{\text{dis}}(A)>2\theta\Longrightarrow \mathcal{N}_{\theta}^{*}\big(\Delta^2(A,x)\big)>0\text{ for all }x\in A
$$
If $d\geq2$, then\footnote{We will actually prove a stronger statement with the exceptional set of bad pins $E_{\theta}(A)=\left\{x\in A: \mathcal{N}_{\theta}^{*}\big(\Delta^2(A,x)\big)=0\right\}$.}\footnote{For \(\theta<\frac12\), the threshold \(2(d-1)\theta\) guarantees a good pin, while the larger threshold \(d-2+2\theta\) guarantees that every pin is good. The two thresholds coincide when \(d=2\).}
$$
\dim_{\mathcal{H}}^{\text{dis}}(A)>
\left\{
\begin{array}{l}
    2\left(d-1\right)\theta \\
    d-2+2\theta 
\end{array}
\right.
\Longrightarrow
\left\{
\begin{array}{l}
     \mathcal{N}_{\theta}^{*}\big(\Delta^2(A,x)\big)>0\text{ for some }x\in A \\
     \mathcal{N}_{\theta}^{*}\big(\Delta^2(A,x)\big)>0\text{ for all }x\in A
\end{array}
\right.
$$
or equivalently, if $d=1$, then\footnote{We will actually show that $\dim_{\mathcal{UM}}^{\text{dis}}\big(\Delta^2(A,x)\big)=\frac{\dim_{\mathcal{UM}}^{\text{dis}}(A)}{2}$ for all $x\in A$ when $d=1,2$.}
$$
\dim_{\mathcal{UM}}^{\text{dis}}\big(\Delta^2(A,x)\big)\geq\frac{\dim_{\mathcal{H}}^{\text{dis}}(A)}{2}\text{ for all }x\in A
$$
If $d\geq2$, then\footnote{The lower bound \(\frac{\dim_{\mathcal{H}}^{\mathrm{dis}}(A)}{2(d-1)}\) holds for some pin, while \(\frac{\dim_{\mathcal{H}}^{\mathrm{dis}}(A)-(d-2)}{2}\) holds for every pin. For \(d\ge3\), the former is larger when \(\dim_{\mathcal{H}}^{\mathrm{dis}}(A)<d-1\), and the latter is larger when \(\dim_{\mathcal{H}}^{\mathrm{dis}}(A)>d-1\). The two bounds coincide when \(\dim_{\mathcal{H}}^{\mathrm{dis}}(A)=d-1\) or \(d=2\).}
$$
\dim_{\mathcal{UM}}^{\text{dis}}\big(\Delta^2(A,x)\big)\geq\left\{
\begin{array}{l}
    \frac{\dim_{\mathcal{H}}^{\text{dis}}(A)}{2(d-1)}\text{ for some }x\in A \\
    \frac{\dim_{\mathcal{H}}^{\text{dis}}(A)-(d-2)}{2}\text{ for all }x\in A 
\end{array}
\right.
$$
In particular, Conjecture \ref{conjpin} is true for $d=1,2$.
\end{thm}
\begin{remark}\label{upperboundpin}
We note the following elementary upper bound\footnote{Example \ref{sharp} shows that this bound is sharp for every pin in dimensions \(d=1,2\), and also in dimension \(d=3\) when \(0<\theta\leq \frac{1}{2}\).} for pinned distance sets: for every $A\subseteq\Z^d$ and every $x\in A$, one has
$$
\dim_{\mathcal{UM}}^{\text{dis}}\big(\Delta^2(A,x)\big)\leq\min\left\{\frac{\dim_{\mathcal{UM}}^{\text{dis}}(A)}{2},1\right\}
$$
\end{remark}
\begin{remark}
Note that pinned distance results can imply the same unpinned distance results; however, the bounds in Theorem \ref{mainthm2} are better than the bounds in Theorem \ref{pinmainthm2}.
\end{remark}
\begingroup
\def\PinnedCmpDimension{5} 

\pgfmathsetmacro{\PinnedCmpSomeSlope}{2*(\PinnedCmpDimension-1)}
\pgfmathsetmacro{\PinnedCmpAllIntercept}{\PinnedCmpDimension-2}
\pgfmathsetmacro{\PinnedCmpThetaCap}{
    \PinnedCmpDimension/\PinnedCmpSomeSlope
}
\pgfmathsetmacro{\PinnedCmpCrossDimension}{\PinnedCmpDimension-1}
\pgfmathsetmacro{\PinnedCmpThresholdTop}{
    max(\PinnedCmpDimension+0.7,1.03*\PinnedCmpSomeSlope+0.3)
}
\pgfmathsetmacro{\PinnedCmpDimensionRight}{\PinnedCmpDimension+0.25}

\pgfplotsset{
    pinnedcmpaxis/.style={
        width=12cm,
        height=8cm,
        axis lines=left,
        axis line style={black},
        axis on top,
        clip=true,
        tick label style={font=\small},
        label style={font=\small},
        legend style={
            at={(0.04,0.96)},
            anchor=north west,
            draw=black,
            fill=white,
            fill opacity=1,
            text opacity=1,
            font=\small,
            legend columns=1
        },
        legend cell align=left,
        enlarge x limits=false
    },
    pinnedcmpsome/.style={red!85!black,thick},
    pinnedcmpall/.style={blue!85!black,thick},
    pinnedcmpconjecture/.style={green!50!black,thick},
    pinnedcmpcap/.style={black,thin},
    pinnedcmpguide/.style={gray,dashed,forget plot},
    pinnedcmpmarker/.style={
        only marks,mark=*,mark size=2.2pt,black,forget plot
    }
}

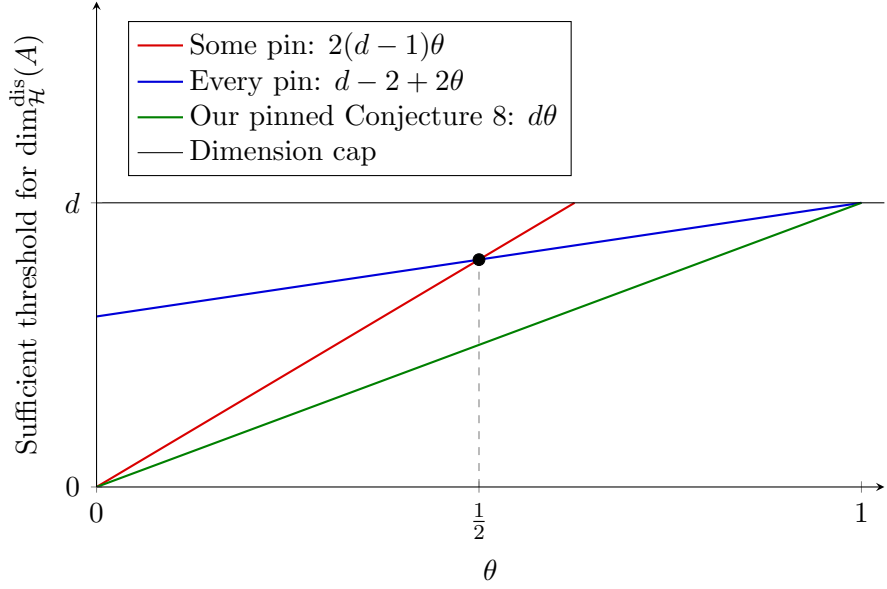
\begin{figure}[H]
\centering
\begin{tikzpicture}
\begin{axis}[
    pinnedcmpaxis,
    xmin=0,xmax=1.03,
    ymin=0,ymax=\PinnedCmpThresholdTop,
    xlabel={$\theta$},
    ylabel={Sufficient threshold for $\dim_{\mathcal H}^{\mathrm{dis}}(A)$},
    xtick={0,0.5,1},
    xticklabels={$0$,$\frac12$,$1$},
    ytick={0,\PinnedCmpDimension},
    yticklabels={$0$,$d$}
]

\addplot[pinnedcmpsome,domain=0:\PinnedCmpThetaCap,samples=2]
    {\PinnedCmpSomeSlope*x};
\addlegendentry{Some pin: $2(d-1)\theta$}

\addplot[pinnedcmpall,domain=0:1,samples=2]
    {\PinnedCmpAllIntercept+2*x};
\addlegendentry{Every pin: $d-2+2\theta$}

\addplot[pinnedcmpconjecture,domain=0:1,samples=2]
    {\PinnedCmpDimension*x};
\addlegendentry{Our pinned Conjecture \ref{conjpin}: $d\theta$}

\addplot[pinnedcmpcap,domain=0:1.03,samples=2]
    {\PinnedCmpDimension};
\addlegendentry{Dimension cap}

\addplot[pinnedcmpguide] coordinates {
    (0.5,0) (0.5,\PinnedCmpCrossDimension)
};
\addplot[pinnedcmpmarker] coordinates {
    (0.5,\PinnedCmpCrossDimension)
};

\end{axis}
\end{tikzpicture}
\caption{Comparison of the sufficient dimension thresholds in our Theorem \ref{pinmainthm2}, the threshold
in our pinned Conjecture \ref{conjpin}, and the dimension cap, for
$d=\PinnedCmpDimension$.}
\label{fig:pinned-comparison-density}
\end{figure}

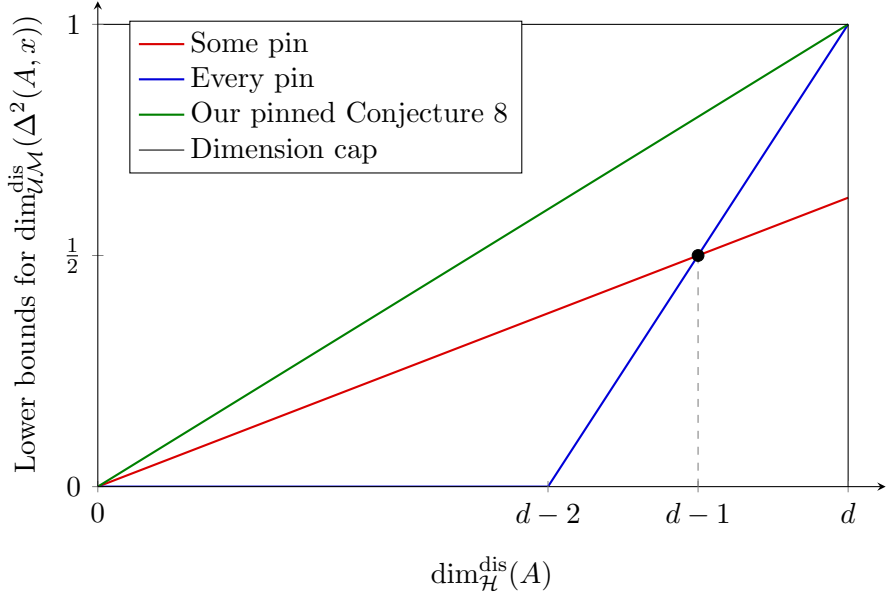
\begin{figure}[H]
\centering
\begin{tikzpicture}
\begin{axis}[
    pinnedcmpaxis,
    xmin=0,xmax=\PinnedCmpDimensionRight,
    ymin=0,ymax=1.05,
    xlabel={$\dim_{\mathcal H}^{\mathrm{dis}}(A)$},
    ylabel={Lower bounds for
        $\dim_{\mathcal{UM}}^{\mathrm{dis}}(\Delta^2(A,x))$},
    xtick={0,\PinnedCmpDimension},
    xticklabels={$0$,$d$},
    ytick={0,0.5,1},
    yticklabels={$0$,$\frac12$,$1$},
    extra x ticks={\PinnedCmpAllIntercept,\PinnedCmpCrossDimension},
    extra x tick labels={$d-2$,$d-1$}
]

\addplot[pinnedcmpsome,domain=0:\PinnedCmpDimension,samples=2]
    {x/\PinnedCmpSomeSlope};
\addlegendentry{Some pin}

\addplot[pinnedcmpall,domain=0:\PinnedCmpAllIntercept,samples=2]
    {0};
\addlegendentry{Every pin}
\addplot[pinnedcmpall,forget plot,
    domain=\PinnedCmpAllIntercept:\PinnedCmpDimension,samples=2]
    {(x-\PinnedCmpAllIntercept)/2};

\addplot[pinnedcmpconjecture,domain=0:\PinnedCmpDimension,samples=2]
    {x/\PinnedCmpDimension};
\addlegendentry{Our pinned Conjecture \ref{conjpin}}

\addplot[pinnedcmpcap,domain=0:\PinnedCmpDimension,samples=2] {1};
\addlegendentry{Dimension cap}
\addplot[pinnedcmpcap,forget plot] coordinates {
    (\PinnedCmpDimension,0) (\PinnedCmpDimension,1)
};

\addplot[pinnedcmpguide] coordinates {
    (\PinnedCmpCrossDimension,0) (\PinnedCmpCrossDimension,0.5)
};
\addplot[pinnedcmpmarker] coordinates {
    (\PinnedCmpCrossDimension,0.5)
};

\end{axis}
\end{tikzpicture}
\caption{Comparison of the pinned dimension lower bounds in our Theorem \ref{pinmainthm2} $\frac{\dim_{\mathcal H}^{\mathrm{dis}}(A)}{2(d-1)}$ for some pin and
$\frac{\dim_{\mathcal H}^{\mathrm{dis}}(A)-(d-2)}{2}$
for every pin, the conjectured lower bound
$\frac{\dim_{\mathcal H}^{\mathrm{dis}}(A)}{d}$, and the dimension cap,
for $d=\PinnedCmpDimension$.}
\label{fig:pinned-comparison-dimension}
\end{figure}

\endgroup
\subsection{Organization}
Section \ref{M} introduces generalized densities and discrete dimensions, discusses affine independence and dimension, and states our main results in a more general setting, including estimates for the exceptional set of pins, together with an analysis of the relative strengths of the resulting bounds. Section \ref{MT} reviews the main tools used in this paper, including elementary integer sphere counting, \(\ell^p\)-improving inequalities for discrete spherical averages, and results on the Erd\H{o}s distinct distances problem. Section \ref{P} provides the proofs of our main theorems for both the unpinned and pinned distance problems. Section \ref{E} presents examples illustrating the sharpness of both conjectures in low dimensions and the different roles of discrete Hausdorff dimension, upper discrete mass dimension, and upper discrete Banach dimension in our unpinned and pinned distance problems. These examples also show that none of these dimensions alone provides a nontrivial upper bound for the upper discrete mass dimension of unpinned distance sets, while discrete Hausdorff dimension alone is insufficient in the pinned setting. Section \ref{D} explores connections between our unpinned and pinned problems and several related results and techniques, including \(\ell^p\)-improving inequalities for discrete spherical averages, finite unpinned and pinned distance problems over the reals and finite fields, homogeneity, \(s\)-adaptability, and discrete analogs of Frostman's theory and energy integrals. We discuss the limitations of these approaches in our setting, including the difficulty of obtaining a fixed pin from finite pinned distance estimates. Finally, Section \ref{A} establishes the lower bound for the discrete Hausdorff dimension of Example \ref{sharp} in all dimensions using discrete capacity estimates, completing the analysis of our sharp example in low dimensions.
\section{Main Results}\label{M}
\subsection{Densities}
To measure subsets of $\Z^d$, the notions of densities are essential. There is some similarity between densities and measures such as subadditivity and translation invariance; see \cite{buck1946measure,leonetti2020notions} for more on their similarity. Motivated by the fractal structure of subsets of $\Z^d$, we generalize not only the upper natural density but also the upper Banach density:

{We recall that $Q_N$ is the discrete cube defined by
$$
Q_N=\left[-\frac{N}{2},\frac{N}{2}\right]^d\cap\Z^d
$$}
\begin{defn}
    Let $d\in\N$ and $\theta\in[0,1]$. The $\bm{\theta}$\textbf{-upper natural density} of $A\subseteq\Z^d$ is defined by
    $$
    \mathcal{N}_{\theta}^{*}(A)=\limsup_{N\to\infty}\frac{|A\cap Q_N|}{|Q_N|^{\theta}}
    $$
\end{defn}
One can see that the $\theta$-upper natural density is the usual upper natural density when $\theta=1$. We can use the $\theta$-upper natural density to define the upper discrete mass dimension of $A$.
\begin{defn}
    Let $d\in\N$ and $\theta\in[0,1]$. The $\bm{\theta}$\textbf{-upper Banach density} of $A\subseteq\Z^d$ is defined by
    $$
    \delta_{\theta}^{*}(A)=\limsup_{N\to\infty}\sup_{t\in\Z^d}\frac{|A\cap(t+Q_N)|}{|Q_N|^{\theta}}
    $$
\end{defn}
One can see that the $\theta$-upper Banach density is the usual upper Banach density when $\theta=1$. We will make use of the $\theta$-upper Banach density to define a weaker notion of upper discrete mass dimension.
\subsection{Discrete Dimensions}
To study the fractal structure of subsets of $\Z^d$, the language of discrete dimensions is useful. There are extensive and systematic studies on discrete dimensions in the literature (see, e.g., \cite{barlow1992defining,bishop2017fractals,glasscock2024additive}), with numerous applications—particularly in probability theory, such as Brownian motion and Markov chains (see \cite{bishop2017fractals} for a detailed discussion of their relationship). We recall the upper discrete mass dimension and the discrete Hausdorff dimension here:
\begin{defn}
    Let $d\in\N$ and $A\subseteq\Z^d$. The \textbf{upper discrete mass dimension} of $A$ is defined as
    $$
    \dim_{\mathcal{UM}}^{\text{dis}}(A)=d\cdot\sup\left\{\theta\in[0,1]:\mathcal{N}_{\theta}^{*}(A)>0\right\}
    $$
\end{defn}
\begin{defn}
    Let $d\in\N$ and $A\subseteq\Z^d$. The \textbf{discrete Hausdorff dimension} of $A$ is defined as
    $$
    \dim_{\mathcal{H}}^{\text{dis}}(A)=\inf\big\{\alpha>0:\mathscr{H}_{\alpha}(A)<\infty\big\}
    $$
    where $\mathscr{H}_{\alpha}(A)$ is the \textbf{discrete Hausdorff content} of $A$ defined by
    $$
    \mathscr{H}_{\alpha}(A)=\inf\sum_{Q}\left(\frac{\text{diam}(Q)}{\displaystyle\max_{x\in Q}\|x\|}\right)^{\alpha}
    $$
    where the infimum is over all coverings of $A$ by discrete cubes $Q$.
\end{defn}
There is an elementary property in \cite{barlow1992defining}:
\begin{prop}\label{pfd}
    For every $A\subseteq\Z^d$, one has $\dim_{\mathcal{H}}^{\text{dis}}(A)\leq\dim_{\mathcal{UM}}^{\text{dis}}(A)$.
\end{prop}
Now we define a translation-uniform counterpart of the upper discrete mass dimension based on the $\theta$-upper Banach density to state our main results.
\begin{defn}
    Let $d\in\N$ and $A\subseteq\Z^d$. The \textbf{upper discrete Banach dimension}\footnote{For subsets of $\Z^d$, the upper discrete Banach dimension coincides with the counting dimension studied in \cite{lima2014marstrand,glasscock2016marstrand,pathak2020marstrand}. We use this terminology to emphasize its formulation in terms of generalized upper Banach densities.} of $A$ is defined by
    \begin{align*}
        \dim_{\mathcal{UB}}^{\text{dis}}(A)&=d\cdot\sup\big\{\theta\in[0,1]:\delta_{\theta}^{*}(A)>0\big\}
    \end{align*}
\end{defn}
\begin{remark}\label{dr}
By definition, for every $A\subseteq\Z^d$, one has $\dim_{\mathcal{UM}}^{\text{dis}}(A)\leq\dim_{\mathcal{UB}}^{\text{dis}}(A)$. Together with Proposition \ref{pfd}, one has
\begin{equation}\label{inee}
\dim_{\mathcal{H}}^{\text{dis}}(A)\leq\dim_{\mathcal{UM}}^{\text{dis}}(A)\leq\dim_{\mathcal{UB}}^{\text{dis}}(A)
\end{equation}
See the below Example \ref{peg} and Example \ref{manypin}. One can find that the inequalities in \eqref{inee} are possibly strict for some particular $A\subseteq\Z^d$.
\end{remark}
{\subsection{Affine Independence and Dimension}For $E\subseteq\R^d$, let $\mathrm{aff}(E)$ denote the smallest affine subspace of $\R^d$ containing $E$, and let $\dim_{\R}\left(\mathrm{aff}(E)\right)$ denote its \textbf{affine dimension} in linear algebra. Points $x_0,\ldots,x_m\in\R^d$ are said to be \textbf{affinely independent} if and only if $x_1-x_0,\ldots,x_m-x_0$ are linearly independent.
\begin{prop}
    For $E\subseteq\R^d$, one has
    $$
    \dim_{\R}\left(\mathrm{aff}(E)\right)\geq m\Longleftrightarrow E\text{ contains }m+1\text{ affinely independent points.}
    $$
\end{prop}
\begin{remark}\label{affrd}
    If $A\subseteq\Z^d$, then $\mathrm{aff}(A)$ is an integer affine subspace of $\R^d$. Moreover, one has
    $$
    \dim_{\mathcal{UM}}^{\text{dis}}(A)\leq\dim_{\mathcal{UB}}^{\text{dis}}(A)\leq\dim_{\R}\left(\mathrm{aff}(A)\right)\leq d
    $$
\end{remark}
\begin{remark}\label{affrk}
    Let $A\subseteq\Z^d$ and suppose that
    $$
    r=\dim_{\R}\left(\mathrm{aff}(A)\right)<d
    $$
    Although $A$ is contained in an $r$-dimensional affine subspace of $\R^d$, one cannot in general identify the distance problem for $A$ with the standard distance problem in $\Z^r$: fix $x_0\in A$ and let
    $$
    \Lambda=\Z^d\cap\left(\mathrm{aff}(A)-x_0\right)
    $$
    Then $\Lambda$ is a rank-$r$ lattice. Choose a $\Z$-basis $v_1,\ldots,v_r$ of $\Lambda$. Every $x\in A$ can be written as
    $$
    x=x_0+Vk=x_0+\sum_{j=1}^rk_jv_j
    $$
    where $V=(v_1,\ldots,v_r)\in\mathrm{M}_{d\times r}(\Z)$ and $k\in\Z^r$. Thus, $A$ may be identified algebraically with a subset of $\Z^r$. However, under this identification the Euclidean distance between $x,y\in A$ is given by
    $$
    \|x-y\|^2=(k-l)^TV^TV(k-l)=Q(k-l)
    $$
    where $Q(z)=z^TV^TVz$ is a positive-definite integral quadratic form. In general, $Q(z)\neq\|z\|^2$. Hence, Euclidean spheres in $\mathrm{aff}(A)\cap\Z^d$ correspond to the quadratic level sets $\left\{z\in\Z^r:Q(z)=\lambda\right\}$ rather than to the standard discrete spheres $\left\{z\in\Z^r:\|z\|^2=\lambda\right\}$.
\end{remark}}
\subsection{Statement of Results}
Recall the elementary sphere bound
$$
T_{\mathrm{Sph}}(r,\theta)=\left\{
\begin{array}{ll}
    2\theta &\text{if }r=1,2  \\
    r-2+2\theta &\text{if }r\geq3
\end{array}
\right.
$$
and the finite Erd\H{o}s distance bound 
$$
T_{\mathrm{Erd}}(r,\theta)=\left\{
\begin{array}{ll}
    2\theta &\text{if }r=1,2  \\
    \frac{r^2+2r+c_r}{r+1}\theta &\text{if }r\geq3
\end{array}
\right.
\text{ where }c_r=
\left\{
\begin{array}{ll}
    -3+\frac{4}{r^2}&\text{if }r\text{ is even}\\
    -3&\text{if }r\text{ is odd}
\end{array}
\right.
$$
For $d\geq5$, recall the $\ell^p$-improving bound
$$
T_{\ell^p}(d,\theta)=\frac{(d+2)\theta+(d-2)}{2}
$$
The following results strengthen our main theorems:
\begin{thm}\label{mainthm}
Let $A\subseteq\Z^d$ and $0\leq\theta\leq1$. Let $r=\dim_{\R}\left(\mathrm{aff}(A)\right)$. Then
$$
\dim_{\mathcal{UB}}^{\text{dis}}(A)>T_{d,r}(\theta)\Longrightarrow \mathcal{N}_{\theta}^{*}\big(\Delta^2(A)\big)>0
$$
where\footnote{Note that $T_{\ell^p}(d,\theta)\geq T_{\mathrm{Erd}}(r,\theta)$ for all $0\le\theta\le1$ when $5\leq r<d$.}
$$
T_{d,r}(\theta)=\left\{
\begin{array}{ll}
     T_{\mathrm{Sph}}(r,\theta)=T_{\mathrm{Erd}}(r,\theta) &\text{if }r=1,2  \\
     T_{\mathrm{Erd}}(r,\theta)&\text{if }r=3\\
     \min\left\{T_{\mathrm{Sph}}(r,\theta),T_{\mathrm{Erd}}(r,\theta)\right\}&\text{if }r=4\text{ or }5\leq r<d\\
     \min\left\{T_{\ell^p}(d,\theta),T_{\mathrm{Erd}}(r,\theta)\right\}&\text{if }r=d\geq5
\end{array}
\right.
$$
in other words,
$$
T_{d,r}(\theta)=\left\{
\begin{array}{ll}
    2\theta &\text{if }r=1,2  \\
    3\theta &\text{if }r=3\\
    \min\left\{r-2+2\theta,\frac{r^2+2r+c_r}{r+1}\theta\right\}&\text{if }r=4\text{ or }5\leq r<d\\
    \min\left\{\frac{(d+2)\theta+(d-2)}{2},\frac{r^2+2r+c_r}{r+1}\theta\right\}&\text{if }r=d\geq5
\end{array}
\right.
$$
or equivalently,\footnote{Example \ref{egggg5555} shows that no nontrivial upper bound for $\dim_{\mathcal{UM}}^{\mathrm{dis}}\big(\Delta^2(A)\big)$ can depend solely on $\dim_{\mathcal H}^{\mathrm{dis}}(A)$, $\dim_{\mathcal{UM}}^{\text{dis}}(A)$, or $\dim_{\mathcal{UB}}^{\text{dis}}(A)$.}
$$
\dim_{\mathcal{UM}}^{\text{dis}}\big(\Delta^2(A)\big)\geq
\left\{
\begin{array}{ll}
\frac{\dim_{\mathcal{UB}}^{\text{dis}}(A)}{2}&\text{if }r=1,2\\
\frac{\dim_{\mathcal{UB}}^{\text{dis}}(A)}{3}&\text{if }r=3\\
\max\left\{\frac{\dim_{\mathcal{UB}}^{\text{dis}}(A)-(r-2)}{2},\frac{\dim_{\mathcal{UB}}^{\text{dis}}(A)(r+1)}{r^2+2r+c_r}\right\}&\text{if }r=4\text{ or }5\leq r<d\\
\max\left\{\frac{2\dim_{\mathcal{UB}}^{\text{dis}}(A)-d+2}{d+2},\frac{\dim_{\mathcal{UB}}^{\text{dis}}(A)(r+1)}{r^2+2r+c_r}\right\}&\text{if }r=d\geq5
\end{array}
\right.
$$
\end{thm}
Define the exceptional set of bad pins:
$$
E_{\theta}(A)=\left\{x\in A: \mathcal{N}_{\theta}^{*}\big(\Delta^2(A,x)\big)=0\right\}
$$
\begin{thm}\label{pinmainthm}
Let $A\subseteq\Z^d$ and $0\leq\theta\leq1$. Let $r=\dim_{\R}\left(\mathrm{aff}(A)\right)$. If $r=1$, then
$$
\dim_{\mathcal{UM}}^{\text{dis}}(A)>2\theta\Longrightarrow E_{\theta}(A)=\varnothing
$$
or equivalently,\footnote{Together with Remark \ref{upperboundpin}, this yields $\dim_{\mathcal{UM}}^{\text{dis}}\big(\Delta^2(A,x)\big)=\frac{\dim_{\mathcal{UM}}^{\text{dis}}(A)}{2}$ for all $x\in A$ when $r=1,2$.}
$$
\dim_{\mathcal{UM}}^{\text{dis}}\big(\Delta^2(A,x)\big)\geq\frac{\dim_{\mathcal{UM}}^{\text{dis}}(A)}{2}\text{ for all }x\in A
$$
If $r\geq2$, then for every integer $0\leq K\leq r-2$, one has\footnote{Here we adopt the convention: $\dim_{\R}\left(\mathrm{aff}\left(E_{\theta}(A)\right)\right)\leq-1$ if and only if $E_{\theta}(A)=\varnothing$ if and only if for every $x\in A$, one has $\mathcal{N}_{\theta}^{*}\big(\Delta^2(A,x)\big)>0$.}\footnote{The threshold \(r-2+2\theta+K(2\theta-1)\) decreases with \(K\) when \(\theta<\frac12\), equals \(r-1\) for every \(K\) when \(\theta=\frac12\), and increases with \(K\) when \(\theta>\frac12\). Thus, for \(\theta\ge\frac12\), all these conclusions follow from the \(K=0\) case, which guarantees that every pin is good.}
$$
\dim_{\mathcal{UM}}^{\text{dis}}(A)>r-2+2\theta+K(2\theta-1)\Longrightarrow\dim_{\R}\left(\mathrm{aff}\left(E_{\theta}(A)\right)\right)\leq K-1
$$
or equivalently, among any \(K+1\) affinely independent points
\(x_0,\ldots,x_K\in A\), there exists \(i\in\{0,\ldots,K\}\) such that\footnote{This lower bound is positive precisely when \(\dim_{\mathcal{UM}}^{\mathrm{dis}}(A)>r-K-2\). As \(K\) increases, it is nondecreasing when \(\dim_{\mathcal{UM}}^{\mathrm{dis}}(A)<r-1\), equals \(\frac12\) for every \(K\) when \(\dim_{\mathcal{UM}}^{\mathrm{dis}}(A)=r-1\), and is strictly decreasing when \(\dim_{\mathcal{UM}}^{\mathrm{dis}}(A)>r-1\). Thus, when \(\dim_{\mathcal{UM}}^{\mathrm{dis}}(A)\ge r-1\), all these conclusions follow from the \(K=0\) case, which applies to every pin.}
$$
\dim_{\mathcal{UM}}^{\text{dis}}\big(\Delta^2(A,x_i)\big)\geq\frac{\dim_{\mathcal{UM}}^{\text{dis}}(A)-(r-2)+K}{2(K+1)}
$$
In particular, taking \(K=0\) and \(K=r-2\) yields the following conclusions:\footnote{For \(\theta<\frac12\), the threshold \(2(r-1)\theta\) guarantees that \(\dim_{\R}\left(\mathrm{aff}\left(E_{\theta}(A)\right)\right)\leq r-3\), while the larger threshold \(r-2+2\theta\) guarantees that every pin is good. The two thresholds coincide when \(r=2\).}
$$
\dim_{\mathcal{UM}}^{\text{dis}}(A)>
\left\{
\begin{array}{l}
    2(r-1)\theta \\
    r-2+2\theta 
\end{array}
\right.
\Longrightarrow
\left\{
\begin{array}{l}
     \dim_{\R}\left(\mathrm{aff}\left(E_{\theta}(A)\right)\right)\leq r-3   \\
     E_{\theta}(A)=\varnothing
\end{array}
\right.
$$
or equivalently,\footnote{The lower bound \(\frac{\dim_{\mathcal{UM}}^{\mathrm{dis}}(A)}{2(r-1)}\) holds for at least one pin among any \(r-1\) affinely independent points of \(A\), while \(\frac{\dim_{\mathcal{UM}}^{\mathrm{dis}}(A)-(r-2)}{2}\) holds for every pin. For \(r\ge3\), the former is larger when \(\dim_{\mathcal{UM}}^{\mathrm{dis}}(A)<r-1\), and the latter is larger when \(\dim_{\mathcal{UM}}^{\mathrm{dis}}(A)>r-1\). The two bounds coincide when \(\dim_{\mathcal{UM}}^{\mathrm{dis}}(A)=r-1\) or \(r=2\).}
\begin{enumerate}[label=(\alph*)]
    \item Among any $r-1$ affinely independent points $x_1,\ldots,x_{r-1}\in A$, there exists $i\in\{1,\ldots,r-1\}$ such that
    $$
    \dim_{\mathcal{UM}}^{\text{dis}}\big(\Delta^2(A,x_i)\big)\geq\frac{\dim_{\mathcal{UM}}^{\text{dis}}(A)}{2(r-1)}
    $$
    \item One has
    $$
    \dim_{\mathcal{UM}}^{\text{dis}}\big(\Delta^2(A,x)\big)\geq\frac{\dim_{\mathcal{UM}}^{\text{dis}}(A)-(r-2)}{2}\text{ for all }x\in A
    $$
\end{enumerate}
\end{thm}
By Remark \ref{dr} and Remark \ref{affrd}, one can see that Theorem \ref{mainthm} and Theorem \ref{pinmainthm} imply Theorem \ref{mainthm2} and Theorem \ref{pinmainthm2}, respectively.

One can see that the elementary sphere bound is better than the finite Erd\H{o}s distance bound when $\theta\geq\theta_{\mathrm{stronger}}^{\mathrm{(Sph,Erd)}}$ or equivalently $\dim_{\mathcal{UB}}^{\text{dis}}(A)\geq d_{\mathrm{stronger}}^{\mathrm{(Sph,Erd)}}$, where
$$
\theta_{\mathrm{stronger}}^{\mathrm{(Sph,Erd)}}=\frac{(r-2)(r+1)}{r^2+c_r-2}
\quad\quad\quad d_{\mathrm{stronger}}^{\mathrm{(Sph,Erd)}}=\frac{(r-2)(r^2+2r+c_r)}{r^2+c_r-2}
$$
\def\rVal{5}

\pgfmathsetmacro{\cVal}{
    -3+(1-mod(\rVal,2))*4/(\rVal*\rVal)
}
\pgfmathsetmacro{\CVal}{
    (\rVal*\rVal+2*\rVal+\cVal)/(\rVal+1)
}
\pgfmathsetmacro{\thetaCross}{
    (\rVal-2)/(\CVal-2)
}
\pgfmathsetmacro{\thetaInapp}{
    \rVal/\CVal
}
\pgfmathsetmacro{\sCross}{
    \CVal*\thetaCross
}
\pgfmathsetmacro{\yMaxOne}{\rVal+0.6}
\pgfmathsetmacro{\xMaxTwo}{\rVal+0.2}


\begin{figure}[H]
\centering

\begin{tikzpicture}
\begin{axis}[
    width=12cm,
    height=8cm,
    xmin=0,
    xmax=1.03,
    ymin=0,
    ymax=\yMaxOne,
    axis lines=left,
    xlabel={$\theta$},
    ylabel={
        Sufficient threshold for
        $\dim_{\mathcal{UB}}^{\mathrm{dis}}(A)$
    },
    xtick={0,1},
    xticklabels={$0$,$1$},
    ytick={0,\rVal},
    yticklabels={$0$,$r$},
    extra x ticks={\thetaCross,\thetaInapp},
    extra x tick labels={
        $\theta_{\mathrm{stronger}}^{\mathrm{(Sph,Erd)}}$,
        $\theta_{\mathrm{inapplicable}}^{\mathrm{Erd}}$
    },
    extra x tick style={
        tick label style={
            rotate=90,
            anchor=east,
            font=\scriptsize,
            yshift=-2pt
        }
    },
    tick label style={font=\small},
    legend style={
        at={(0.05,0.95)},
    anchor=north west,
    draw=black,
    fill=white,
    fill opacity=1,
    text opacity=1,
    font=\small,
    legend columns=1
    },
    legend cell align=left,
]

\addplot[
    blue!85!black,
    thick,
    domain=0:1,
    samples=2
] {\CVal*x};
\addlegendentry{Finite Erd\H{o}s distance bound}

\addplot[
    red!85!black,
    thick,
    domain=0:1,
    samples=2
] {\rVal-2+2*x};
\addlegendentry{Elementary sphere bound}

\addplot[
    green!50!black,
    thick,
    domain=0:1,
    samples=2
] {\rVal*x};
\addlegendentry{Our Conjecture \ref{conj}}

\addplot[
    black,
    thick,
    domain=0:1,
    samples=2
] {\rVal};
\addlegendentry{Dimension cap}

\addplot[
    gray,
    dashed,
    forget plot
] coordinates {
    (\thetaCross,0)
    (\thetaCross,\CVal*\thetaCross)
};

\addplot[
    gray,
    thick,
    dotted,
    forget plot
] coordinates {
    (\thetaInapp,0)
    (\thetaInapp,\rVal)
};

\addplot[
    only marks,
    mark=*,
    mark size=2pt,
    black,
    forget plot
] coordinates {
    (\thetaCross,\CVal*\thetaCross)
    (\thetaInapp,\rVal)
};

\end{axis}
\end{tikzpicture}

\caption{
Comparison of the lower bounds for
$\dim_{\mathcal{UB}}^{\mathrm{dis}}(A)$
obtained from the finite Erd\H{o}s distance bound
and the elementary sphere bound,
together with our Conjecture \ref{conj} and
the dimension cap, for $r=\rVal$.
}
\label{fig:erdos-sphere-theta}

\end{figure}


\begin{figure}[H]
\centering

\begin{tikzpicture}
\begin{axis}[
    width=12cm,
    height=8cm,
    xmin=0,
    xmax=\xMaxTwo,
    ymin=0,
    ymax=1.08,
    axis lines=left,
    xlabel={
        $\dim_{\mathcal{UB}}^{\mathrm{dis}}(A)$
    },
    ylabel={
        Lower bounds for
        $\dim_{\mathcal{UM}}^{\mathrm{dis}}
        (\Delta^2(A))$
    },
    xtick={0,\rVal},
    xticklabels={$0$,$r$},
    ytick={0,1},
    yticklabels={$0$,$1$},
    extra x ticks={\sCross},
    extra x tick labels={
        $d_{\mathrm{stronger}}^{\mathrm{(Sph,Erd)}}$
    },
    extra x tick style={
        tick label style={
            rotate=90,
            anchor=east,
            font=\scriptsize,
            yshift=-2pt
        }
    },
    tick label style={font=\small},
    legend style={
        at={(0.05,0.95)},
    anchor=north west,
    draw=black,
    fill=white,
    fill opacity=1,
    text opacity=1,
    font=\small,
    legend columns=1
    },
    legend cell align=left,
]

\addplot[
    blue!85!black,
    thick,
    domain=0:\rVal,
    samples=2
] {x/\CVal};
\addlegendentry{Finite Erd\H{o}s distance bound}

\addplot[
    black,
    thick,
    domain=0:\rVal,
    samples=2
] {1};
\addlegendentry{Dimension cap}
\addplot[
    red!85!black,
    thick,
    domain=\rVal-2:\rVal,
    samples=2
] {(x-\rVal+2)/2};
\addlegendentry{Elementary sphere bound}

\addplot[
    green!50!black,
    thick,
    domain=0:\rVal,
    samples=2
] {x/\rVal};
\addlegendentry{Our Conjecture \ref{conj}}

\addplot[
    black,
    thick,
    forget plot
] coordinates {
    (\rVal,0)
    (\rVal,1)
};

\addplot[
    gray,
    dashed,
    forget plot
] coordinates {
    (\sCross,0)
    (\sCross,\thetaCross)
};

\addplot[
    only marks,
    mark=*,
    mark size=2pt,
    black,
    forget plot
] coordinates {
    (\sCross,\thetaCross)
};

\end{axis}
\end{tikzpicture}

\caption{
Comparison of the lower bounds for
$\dim_{\mathcal{UM}}^{\mathrm{dis}}
(\Delta^2(A))$
obtained from the finite Erd\H{o}s distance bound
and the elementary sphere bound,
together with our Conjecture \ref{conj} and
the dimension cap, for $r=\rVal$.
}
\label{fig:erdos-sphere-dimension}

\end{figure}
\section{Main Tools}\label{MT}
{Our main tools are elementary integer sphere counting, $\ell^p$-improving inequalities for the discrete spherical averages, and the results on the Erd\H{o}s distance problem.
\subsection{Elementary Integer Sphere Counting}
For every $\lambda\in\N$, let
$$S_{\lambda}=\left\{x\in\Z^d:\|x\|^2=\lambda\right\}\quad\quad\quad\sigma_{\lambda}=\bm{1}_{S_{\lambda}}$$
It is well-known that for every $\epsilon>0$, one has (See for example \cite[Chapter~20]{iwaniec2021analytic})
$$
|S_{\lambda}|\lesssim
\left\{
\begin{array}{ll}
    2 & \text{if }d=1 \\
    \lambda^{\frac{d}{2}-1+\epsilon} & \text{if }2\leq d\leq 4 \\
    \lambda^{\frac{d}{2}-1} &\text{if }d\geq5
\end{array}
\right.
$$
We also have the following generalization: for every integer affine subspace $H$ of $\R^d$, one has
$$
\left|H\cap S_{\lambda}\right|\lessapprox
\left\{
\begin{array}{ll}
    1 &\text{if }\dim_{\R}(H)=1 \\
    \lambda^{\frac{\dim_{\R}(H)}{2}-1} & \text{if }\dim_{\R}(H)\geq2
\end{array}
\right.
$$
\subsection{$\ell^p$-Improving Inequalities for Discrete Spherical Averages}
The study of $\ell^p$-improving inequalities for discrete spherical averages is motivated by their continuous counterparts. A classical result on $L^p$-improving inequalities for spherical averages is the following theorem of Littman \cite{littman1973lp}:
\begin{thm}[Littman]
Let $d\geq2$ and $f\colon\R^d\to\C$. Then
$$
\|f\ast d\sigma\|_{L^{d+1}(\R^d)}\lesssim\|f\|_{L^{\frac{d+1}{d}}(\R^d)}
$$
where $d\sigma$ is the Euclidean surface measure on the unit sphere $\mathbb{S}^{d-1}$ in $\R^d$.
\end{thm}
$L^p$-improving inequalities for spherical averages are important tools in harmonic analysis and partial differential equations, especially for the wave equation. They also naturally appear in the study of distance problems.

We will apply the following $\ell^p$-improving inequalities for the discrete spherical averages, proved by Han and Yang \cite{han2026sharpellpimprovingestimatesfixedradius}, to prove our main theorem for $d\geq5$.
\begin{thm}[Han and Yang]\label{lp}
    Let $f\colon\Z^d\to\C$. If $d\geq 5$ and $\frac{d+2}{d}\leq p\leq2$, then for every $\epsilon>0$ and every $\lambda\in\N$, one has
    $$
    \|f\ast\sigma_{\lambda}\|_{\ell^{p'}(\Z^d)}\lesssim_{d,p,\epsilon}\lambda^{\epsilon-\frac{d}{2}\left(\frac{2}{p}-1\right)}|S_{\lambda}|\|f\|_{\ell^p(\Z^d)}
    $$
\end{thm}
\subsection{Erd\H{o}s Distinct Distances Problem}
The current best result for $\R^2$ was proved by Guth and Katz \cite{guth2015erdHos}:
\begin{thm}[Guth and Katz]\label{GK}
Let $A\subseteq\R^2$ with $|A|=n$. Then
$$
\left|\Delta^2(A)\right|\gtrsim\frac{n}{\log{n}}
$$
\end{thm}
The recent breakthrough for $\R^3$ was proved by Tidor, Yu, and Zakharov \cite{tidor2026erdhosdistinctdistancesproblem}:
\begin{thm}[Tidor, Yu, and Zakharov]\label{n3}
Let $A\subseteq\R^3$ with $|A|=n$. Then
$$
\left|\Delta^2(A)\right|\gtrapprox n^{\frac{2}{3}}
$$
\end{thm}
Solymosi and Vu \cite{solymosi2008near} showed the following recursion for the Erd\H{o}s distance problem:
\begin{thm}[Solymosi and Vu]\label{svrecur}
Suppose that every finite set $P\subseteq\R^{d_0}$ satisfies $\left|\Delta^2(P)\right|\gtrsim |P|^{\alpha_{d_0}}$. Then for every $d\geq d_0$ and every finite set $A\subseteq\R^d$ with $|A|=n$,
\begin{enumerate}[label=(\alph*)]
    \item One has
    $$
    \left|\Delta^2(A)\right|\gtrsim n^{\frac{2d}{(d+d_0+1)(d-d_0)+\frac{2d_0}{\alpha_{d_0}}}}
    $$
    \item Moreover, if $d-d_0$ is even, then
    $$
    \left|\Delta^2(A)\right|\gtrsim n^{\frac{2(d+1)}{(d+d_0+2)(d-d_0)+\frac{2(d_0+1)}{\alpha_{d_0}}}}
    $$
\end{enumerate}
\end{thm}
The current best result for higher dimensions is applying Theorem \ref{svrecur} to Theorem \ref{n3} (by applying Theorem \ref{svrecur} (b) to Theorem \ref{n3} to get results for odd dimensions $d\geq3$ and then applying Theorem \ref{svrecur} (a) to these results for odd dimensions $d\geq3$ to get results for even dimensions $d+1$):
\begin{thm}[The current best result on the Erd\H{o}s distance problem]\label{sv}
Let $d\geq3$ and $A\subseteq\R^d$ with $|A|=n$. Then
$$
\left|\Delta^2(A)\right|\gtrapprox n^{\frac{2(d+1)}{d^2+2d+c_d}}
$$
where
$$
c_d=
\left\{
\begin{array}{ll}
    -3+\frac{4}{d^2}&\text{if }d\text{ is even} \\
    -3&\text{if }d\text{ is odd}
\end{array}
\right.
$$
\end{thm}
}
\section{Proofs of the Main Results}\label{P}
{\subsection{Proof of Theorem \ref{mainthm}}
First, we prove the $\ell^p$-improving bound: our proof is motivated by the Fourier analytic methods in \cite{magyar2008distance,lyall2020distances}.

For every $f_0,f_1\colon\Z^d\to\C$, we define
$$
T_{(f_0,f_1)}(\lambda)=\sum_{x\in\Z^d}f_0(x)(f_1\ast\sigma_{\lambda})(x)
$$
Suppose that $A\subseteq\Z^d$. For every $N\in\N$ and every $t\in\Z^d$, let
$$
A_{N,t}=A\cap(t+Q_N)\quad\quad\quad\Delta_{d,N}^2(A)=\Delta^2(A)\cap\left\{1,2,\ldots,dN^2\right\}
$$
Our first lemma is a lower bound for the average of this operator when $f_0=f_1=\bm{1}_{A_{N,t}}$.
\begin{lemma}\label{firstlemma}
    Let $A\subseteq\Z^d$. For every $N\in\N$ and every $t\in\Z^d$, one has
    $$
    \sum_{\lambda\in{\Delta_{d,N}^2(A)}}T_{\left(\bm{1}_{A_{N,t}},\bm{1}_{A_{N,t}}\right)}(\lambda)=\left|A_{N,t}\right|\big(\left|A_{N,t}\right|-1\big)
    $$
\end{lemma}
\begin{proof}
For every $x\in A_{N,t}$, we know
{$$
\bigsqcup_{\lambda\in\Delta_{d,N}^2(A)}\big(A_{N,t}\cap(x+S_{\lambda})\big)=A_{N,t}\setminus\{x\}
$$}
Hence, we have
$$
\sum_{\lambda\in \Delta_{d,N}^2(A)}\left(\bm{1}_{A_{N,t}}\ast \sigma_{\lambda}\right)(x)=\sum_{\lambda\in\Delta_{d,N}^2(A)}\left|A_{N,t}\cap(x+S_{\lambda})\right|=\left|A_{N,t}\right|-1
$$
Then
\begin{align*}
\sum_{\lambda\in\Delta_{d,N}^2(A)}T_{\left(\bm{1}_{A_{N,t}},\bm{1}_{A_{N,t}}\right)}(\lambda)&=\sum_{x\in A_{N,t}}\sum_{\lambda\in\Delta_{d,N}^2(A)}\left|A_{N,t}\cap(x+S_{\lambda})\right|=\left|A_{N,t}\right|\big(\left|A_{N,t}\right|-1\big)
\end{align*}
\end{proof}
Our second lemma provides an upper bound for the same average using $\ell^p$-improving inequalities for discrete spherical averages for $d\geq5$.
\begin{lemma}\label{secondlemma}
Let $d\geq5$. For every $\epsilon>0$ and every $\frac{d+2}{d}\leq p\leq2$, one has
$$
\sum_{\lambda\in\Delta_{d,N}^2(A)}T_{\left(\bm{1}_{A_{N,t}},\bm{1}_{A_{N,t}}\right)}(\lambda)\lesssim_{d,p,\epsilon}\left|\Delta_{d,N}^2(A)\right|\left|A_{N,t}\right|^{\frac{2}{p}}N^{\frac{2}{p'}(d-(1-\epsilon)p')}
$$
\end{lemma}
\begin{proof}
By Hölder's inequality and Theorem \ref{lp}, we know\footnote{Note that $\epsilon-\frac{d}{2}\left(\frac{2}{p}-1\right)+\frac{d}{2}-1=\frac{d}{p'}-(1-\epsilon)>0$.}
\begin{align*}
\sum_{\lambda\in\Delta_{d,N}^2(A)}T_{\left(\bm{1}_{A_{N,t}},\bm{1}_{A_{N,t}}\right)}(\lambda)&\leq\left|\Delta_{d,N}^2(A)\right|^{\frac{1}{p}}\left|A_{N,t}\right|^{\frac{1}{p}}\left(\sum_{\lambda\in\Delta_{d,N}^2(A)}\left\|\bm{1}_{A_{N,t}}\ast\sigma_{\lambda}\right\|_{\ell^{p'}(\Z^d)}^{p'}\right)^{\frac{1}{p'}}\\
    &\lesssim_{d,p,\epsilon}\left|\Delta_{d,N}^2(A)\right|^{\frac{1}{p}}\left|A_{N,t}\right|^{\frac{1}{p}}\left\|\bm{1}_{A_{N,t}}\right\|_{\ell^{p}(\Z^d)}\left(\sum_{\lambda\in\Delta_{d,N}^2(A)}\left(\lambda^{\epsilon-\frac{d}{2}\left(\frac{2}{p}-1\right)+\frac{d}{2}-1}\right)^{p'}\right)^{\frac{1}{p'}}\\
    &{\leq\left|\Delta_{d,N}^2(A)\right|^{\frac{1}{p}}\left|A_{N,t}\right|^{\frac{2}{p}}\left(\left|\Delta_{d,N}^2(A)\right|\cdot\max_{\lambda\in\Delta_{d,N}^2(A)}\left(\lambda^{\epsilon-\frac{d}{2}\left(\frac{2}{p}-1\right)+\frac{d}{2}-1}\right)^{p'}\right)^{\frac{1}{p'}}}\\
    &\leq\left|\Delta_{d,N}^2(A)\right|^{\frac{1}{p}}\left|A_{N,t}\right|^{\frac{2}{p}}\left(\left|\Delta_{d,N}^2(A)\right|\cdot(dN^2)^{d-(1-\epsilon)p'}\right)^{\frac{1}{p'}}\\
    &\lesssim_{d,p,\epsilon}\left|\Delta_{d,N}^2(A)\right|\left|A_{N,t}\right|^{\frac{2}{p}}N^{\frac{2}{p'}(d-(1-\epsilon)p')}
\end{align*}
\end{proof}
By Lemma \ref{firstlemma} and \ref{secondlemma}, we immediately obtain the following corollary:
\begin{cor}\label{cordensity}
Let $d\geq5$. For every $\epsilon>0$ and every $\frac{d+2}{d}\leq p\leq2$, one has
$$
\left|\Delta_{d,N}^2(A)\right|\gtrsim_{d,p,\epsilon}\frac{\left|A_{N,t}\right|^{\frac{2}{p'}}}{N^{\frac{2}{p'}(d-(1-\epsilon)p')}}
$$
which implies
$$
\dim_{\mathcal{UB}}^{\text{dis}}(A)>\theta p'+(d-p')\Longrightarrow \mathcal{N}_{\theta}^{*}\big(\Delta^2(A)\big)>0
$$
\end{cor}
The $\ell^p$-improving bound follows immediately from Corollary \ref{cordensity} with $p=\frac{d+2}{d}$.

Second, we show the elementary sphere bound: by the proof in Lemma \ref{firstlemma}, we know
$$
\left|A_{N,t}\right|-1=\sum_{\lambda\in\Delta_{d,N}^2(A)}\left|A_{N,t}\cap(x+S_{\lambda})\right|=\sum_{\lambda\in \Delta_{d,N}^2(A)}\left|(A_{N,t}-x)\cap S_{\lambda}\right|\lessapprox\left|\Delta_{d,N}^2(A)\right|N^{\max\{0,r-2\}}
$$
which implies the elementary sphere bound.

Finally, the finite Erd\H{o}s distance bound follows from applying Theorem \ref{GK} and Theorem \ref{sv} to $A_{N,t}$ with $d$ replaced by $r$ since $A_{N,t}$ lies in an $r$-dimensional affine subspace of $\R^d$.
}
{\subsection{Proof of Theorem \ref{pinmainthm}}
Let $A\subseteq\Z^d$. For every $N\in\N$ and every $x\in\Z^d$, we define
$$
A_N=A\cap Q_N\quad\quad\quad\Delta_{d,N}^2(A,x)=\Delta^2(A,x)\cap\left\{1,2,\ldots,dN^2\right\}
$$
\begin{lemma}\label{ptl}
Let $\dim_{\R}\left(\mathrm{aff}(A)\right)\geq2$. For every integer $0\leq K\leq\dim_{\R}\left(\mathrm{aff}(A)\right)-2$, one has
$$
\dim_{\R}\left(\mathrm{aff}\left(E_{\theta}(A)\right)\right)\geq K\Longrightarrow
\dim_{\mathcal{UM}}^{\text{dis}}(A)\leq\dim_{\R}\left(\mathrm{aff}(A)\right)-2+2\theta+K(2\theta-1)
$$
\end{lemma}
\begin{proof}[Proof of Lemma \ref{ptl}]
There exist $K+1$ affinely independent bad pins $x_0,\ldots,x_K\in E_{\theta}(A)$ such that
\begin{equation}\label{pctd}
    \left|\Delta^2(A,x_i)\cap[0,R]\right|=o(R^{\theta})\text{ for all }0\leq i\leq K
\end{equation}
For $y\in A_N$, we define
$$
\Phi_K(y)=\left(\|y-x_0\|^2,\ldots,\|y-x_K\|^2\right)
$$
Then
$$
A_N=\bigsqcup_{\lambda\in\Phi_K(A_N)}\Phi_K^{-1}(\lambda)
$$
Hence, we know
$$
|A_N|\leq|\Phi_K(A_N)|\cdot\max_{\lambda\in\Phi_K(A_N)}\left|\Phi_K^{-1}(\lambda)\right|
$$
First, we show an upper bound for $|\Phi_K(A_N)|$: for every $0\leq i\leq K$ and every $N\geq2\|x_i\|$, we have
$$
\|y-x_i\|\leq\|y\|+\|x_i\|\leq\sqrt{d}N+\frac{N}{2}\lesssim_d N
$$
Therefore, by \eqref{pctd}, each component of $\Phi_K(y)$ has only $o(N^{2\theta})$ possible values. In other words, we obtain
\begin{equation}\label{pctd2}
    |\Phi_K(A_N)|=o(N^{2\theta(K+1)})
\end{equation}
Second, we prove an upper bound for $\max_{\lambda\in\Phi_K(A_N)}\left|\Phi_K^{-1}(\lambda)\right|$: fix one distance tuple $(\lambda_0,\ldots,\lambda_K)\in\Phi_K(A_N)$. We count integer points $y$ satisfying
$$
\|y-x_i\|^2=\lambda_i\text{ for all }0\leq i\leq K
$$
Subtracting the first equation from the $i$-th gives
$$
2y\cdot(x_i-x_0)=\|x_i\|^2-\|x_0\|^2-\lambda_i+\lambda_0
$$
which is linear in $y$. Since $x_0,\ldots,x_K$ are aﬀinely independent, the vectors $x_1-x_0,\ldots,x_K-x_0$ are linearly independent. Thus, the $K$ equations impose $K$ independent linear constraints inside $\mathrm{aff}(A)$. Consequently, the solutions of these linear equations lie in an aﬀine subspace $W_{\lambda}$ of dimension
$$
\dim_{\R}(W_{\lambda})=\dim_{\R}\left(\mathrm{aff}(A)\right)-K\geq2
$$
Furthermore, we get
$$
\Phi_K^{-1}(\lambda)=W_{\lambda}\cap\left\{y\in A_N:\|y-x_0\|^2=\lambda_0\right\}
$$
which is a $\dim_{\R}(W_{\lambda})$-dimensional lattice sphere, so we know
\begin{equation}\label{pctd3}
\left|\Phi_K^{-1}(\lambda)\right|\lesssim_{\epsilon}N^{\dim_{\R}(W_{\lambda})-2+\epsilon}
\end{equation}
Finally, by \eqref{pctd2} and \eqref{pctd3}, we have
$$
|A_N|=o\left(N^{2\theta(K+1)+\dim_{\R}\left(\mathrm{aff}(A)\right)-K-2+\epsilon}\right)
$$
which implies
$$
\dim_{\mathcal{UM}}^{\text{dis}}(A)\leq\dim_{\R}\left(\mathrm{aff}(A)\right)-2+2\theta+K(2\theta-1)
$$
\end{proof}
\begin{remark}\label{ptrk}
If $\dim_{\R}\left(\mathrm{aff}(A)\right)=1$, then one can easily show that
$$
E_{\theta}(A)\neq\varnothing\Longrightarrow
\dim_{\mathcal{UM}}^{\text{dis}}(A)\leq2\theta
$$
by the following inequality:
$$
|A_N|\leq2\left|\Delta^2(A,x)\cap [0,dN^2]\right|
$$
since $A$ lies on a line so that every squared distance from $x$ corresponds to at most two points of $A$.
\end{remark}
Theorem \ref{pinmainthm} follows immediately from Lemma \ref{ptl} and Remark \ref{ptrk}.
}
\section{Main Examples}\label{E}
Here are several examples for our discrete $\Z^d$-variant of the Erd\H{o}s-Falconer unpinned and pinned distance problems:
\begin{eg}\label{eg1}
This is the trivial example.

Let $d\in\N$ and $A=\Z^d$. Clearly, one can see
$$
\dim_{\mathcal{H}}^{\text{dis}}(A)=\dim_{\mathcal{UM}}^{\text{dis}}(A)=\dim_{\mathcal{UB}}^{\text{dis}}(A)=d
$$
and
$$
\Delta^2(A)=\Delta^2(A,x)\text{ for all }x\in A
$$
By Landau's and Legendre's sum of squares theorems, one has
$$
\mathcal{N}_{\theta}^{*}\big(\Delta^2(A)\big)=\mathcal{N}_{\theta}^{*}\big(\Delta^2(A,x)\big)>0
\left\{
\begin{array}{ll}
    \text{for all }\theta\in\left[0,\frac{1}{2}\right] &\text{if }d=1  \vspace{0.1cm}\\
     \text{for all }\theta\in[0,1) &\text{if }d=2 \vspace{0.1cm}\\
     \text{for all }\theta\in[0,1] &\text{if }d\geq3
\end{array}
\right.
$$
or equivalently,
$$
\dim_{\mathcal{UM}}^{\text{dis}}\big(\Delta^2(A)\big)=\dim_{\mathcal{UM}}^{\text{dis}}\big(\Delta^2(A,x)\big)=
\left\{
\begin{array}{ll}
   \frac{1}{2}  &\text{if }d=1  \vspace{0.1cm}\\
   1  &\text{if }d\geq2
\end{array}
\right.
$$
\end{eg}
\begin{eg}\label{sharp}
This construction is inspired by \cite{Erdos01051946,Falconer_1985} and is regarded as a sharp example of our Conjecture \ref{conj} and Conjecture \ref{conjpin} for $d=1,2$. Moreover, this construction also attains the elementary pinned upper bound in Remark \ref{upperboundpin} for every pin when $d=1,2$, and when $d=3$ with $0<\theta\leq\frac{1}{2}$.

For every $\theta\in(0,1)$, we define
$$
s_{n}=2^{n!}\quad\quad\quad\ell_{n}=\left\lfloor\frac{1}{s_n}\prod_{j=1}^ns_j^{\frac{1}{\theta}}\right\rfloor
$$
Let
$$
A_1=\left\{\ell_1(m_1,\ldots,m_d)\in\Z^d:0\leq m_1,\ldots,m_d\leq s_1-1\right\}
$$
be a lattice containing $s_1^d$ points spaced at $\ell_1$ intervals. Then we recursively define $A_{n}$ to be a self-similar lattice containing $s_{n}^d$ copies of $A_{n-1}$ spaced at $\ell_{n}$ intervals. In other words, we define
$$
A_n=A_{n-1}+(\ell_n+L_{n-1})\{0,1,\ldots,s_n-1\}^{d}
$$
where
$$
L_1=\ell_1(s_1-1)\quad\quad\quad L_{n}=s_nL_{n-1}+\ell_{n}(s_n-1)
$$
be the length of each side of $A_n$.
\begin{center}
\begin{tikzpicture}[scale=0.85]


\begin{scope}

\foreach \x in {0,...,3}{
    \foreach \y in {0,...,3}{
        \node at (\x,\y) {\scalebox{0.7}{$\bullet$}};
    }
}

\pgfmathsetmacro{\Ltop}{3}

\draw[red,thick] (0,\Ltop) -- (3,\Ltop);
\node[red,above] at (1.5,\Ltop) {$L_1$};

\draw[blue,thick]
(0,1.5) ellipse (0.32 and 1.85);

\node[blue,left] at (-0.35,1.5) {$s_1$};

\draw[violet,thick] (3,0) -- (3,1);
\node[violet,right] at (3,0.5) {$\ell_1$};

\node[align=center,font=\small] at (1.5,-2.2)
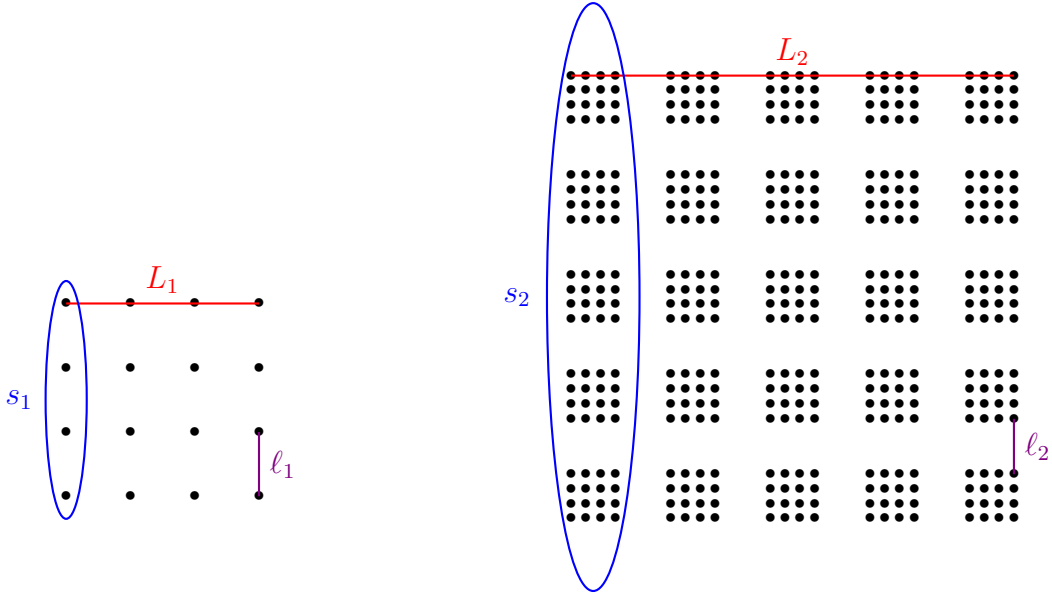
{The lattice $A_1$ for $d=2$};

\end{scope}


\begin{scope}[xshift=8.2cm]

\def\Macro{1.55}
\def\Micro{0.23}

\foreach \i in {0,...,4}{
    \foreach \j in {0,...,4}{
        \foreach \a in {0,...,3}{
            \foreach \b in {0,...,3}{
                \node at
                ({\i*\Macro+(\a-1.5)*\Micro},
                 {\j*\Macro+(\b-1.5)*\Micro})
                {\scalebox{0.7}{$\bullet$}};
            }
        }
    }
}

\pgfmathsetmacro{\Rleft}{-1.5*\Micro}
\pgfmathsetmacro{\Rright}{4*\Macro + 1.5*\Micro}
\pgfmathsetmacro{\Rtop}{4*\Macro + 1.5*\Micro}

\pgfmathsetmacro{\LowerTop}{1.5*\Micro}
\pgfmathsetmacro{\UpperBottom}{\Macro - 1.5*\Micro}

\draw[red,thick] (\Rleft,\Rtop) -- (\Rright,\Rtop);
\node[red,above] at ({(\Rleft+\Rright)/2},\Rtop) {$L_2$};

\draw[blue,thick]
(0,{2*\Macro}) ellipse (0.72 and {2.95*\Macro});

\node[blue,left] at (-0.80,{2*\Macro}) {$s_2$};

\draw[violet,thick]
(\Rright,\LowerTop) -- (\Rright,\UpperBottom);
\node[violet,right] at (\Rright,{(\LowerTop+\UpperBottom)/2}) {$\ell_2$};

\node[align=center,font=\small]
at ({(\Rleft+\Rright)/2},-2.2)
{The self-similar lattice $A_2$ for $d=2$};

\end{scope}

\end{tikzpicture}
\end{center}
Define
$$
A=\bigcup_{n\in\N}A_n
$$
Since we know
$$
|A_n|=s_n^{d+o(1)}\quad\quad\quad L_n=s_n^{\frac{1}{\theta}+o(1)}
$$
it follows that
$$
\dim_{\mathcal{H}}^{\text{dis}}(A)\leq\dim_{\mathcal{UM}}^{\text{dis}}(A)\leq\dim_{\mathcal{UB}}^{\text{dis}}(A)\leq d\theta
$$
We leave the proof of the lower bound $\dim_{\mathcal{H}}^{\text{dis}}(A)\geq d\theta$ in Section \ref{A} to show the equality.

Fix any sufficiently large $N\in\N$, there exists $n\in\N$ such that
$$
\ell_{n}^2\leq N<\ell_{n+1}^2
$$
Then
$$
\left|\Delta^2(A)\cap[0,N]\right|=\left|\Delta^2(A_n)\cap[0,N]\right|
$$
Every difference vector in $A_{n}-A_n$ has the form $D_nu+v$, where
$$
D_n=\ell_n+L_{n-1}\quad\quad\quad u\in\{-(s_n-1),\ldots,(s_n-1)\}^d\quad\quad\quad v\in A_{n-1}-A_{n-1}
$$
If $\|D_nu+v\|^2\leq N$, then
$$
|u_i|\leq\frac{\sqrt{N}+L_{n-1}}{D_n}\lessapprox\frac{\sqrt{N}}{D_n}
$$
Hence, the number of admissible displacement vectors $u$ is at most $\lessapprox q^d$, where
$$
q=\min\left\{s_n,\frac{\sqrt{N}}{D_n}\right\}
$$
That is, we have
\begin{equation}\label{dgl}
    \left|\Delta^2(A_n)\cap[0,N]\right|\lessapprox |A_{n-1}-A_{n-1}|\cdot q^d\leq |A_{n-1}|^2\cdot(s_n^d)^{1-\theta}\left(\frac{\sqrt{N}}{D_n}\right)^{d\theta}\lessapprox s_n^{o(1)}N^{\frac{d\theta}{2}}=N^{\frac{d\theta}{2}+o(1)}
\end{equation}
which implies 
$$
\mathcal{N}_{\theta'}^{*}\big(\Delta^2(A)\big)=0 \text{ for all }\theta'>\frac{d\theta}{2}
$$
In other words, we have
\begin{equation}\label{d2l}
\dim_{\mathcal{UM}}^{\text{dis}}\big(\Delta^2(A)\big)\leq\frac{d\theta}{2}
\end{equation}
Fix $x\in A$. Then there exists $n_x\in\N$ such that $x\in A_{n_x}$. For every $n>n_{x}$, we know
$$
x+D_n\{0,\ldots,s_n-1\}^d\subseteq A_n
$$
which implies
$$
D_n^2\{u_1^2+\cdots+u_d^2:0\leq u_i\leq s_n-1\}\subseteq\Delta^2(A_n,x)
$$
Since we have
$$
\left|\{u_1^2+\cdots+u_d^2:0\leq u_i\leq s_n-1\}\right|=\left\{
\begin{array}{ll}
    s_n &\text{if }d=1  \\
    s_n^{2-o(1)} &\text{if }d=2 
\end{array}
\right.
$$
for every $\theta'<\theta$, we obtain
$$
\mathcal{N}_{\theta'}^{*}\big(\Delta^2(A,x)\big)\geq\limsup_{n\to\infty}\frac{\left|\Delta^2(A_n,x)\right|}{(dL_n^2)^{\theta'}}\gtrsim_x\left\{
\begin{array}{ll}
    \displaystyle\limsup_{n\to\infty}\frac{s_n}{s_n^{2\cdot\frac{\theta'}{\theta}+o(1)}}=\infty &\text{for all }\theta'<\displaystyle\frac{\theta}{2}\text{ if }d=1  \\
    \displaystyle\limsup_{n\to\infty}\frac{s_n^{2-o(1)}}{s_n^{2\cdot\frac{\theta'}{\theta}+o(1)}}=\infty &\text{for all }\theta'<\theta\text{ if }d=2
\end{array}
\right.
$$
which implies
\begin{equation}\label{pinlowerboundeq}
    \dim_{\mathcal{UM}}^{\text{dis}}\big(\Delta^2(A,x)\big)\geq\left\{
\begin{array}{ll}
    \frac{\theta}{2} &\text{if }d=1  \\
    \theta &\text{if }d=2 
\end{array}
\right.
\end{equation}
Therefore, by \eqref{d2l} and \eqref{pinlowerboundeq}, we get
$$
\dim_{\mathcal{UM}}^{\text{dis}}\big(\Delta^2(A)\big)=\dim_{\mathcal{UM}}^{\text{dis}}\big(\Delta^2(A,x)\big)=\left\{
\begin{array}{ll}
    \frac{\theta}{2} &\text{if }d=1  \\
    \theta &\text{if }d=2 
\end{array}
\right.
$$
for all $x\in A$.
\end{eg}
\begin{remark}
In higher dimensions \(d\ge3\), the unpinned and pinned distance sets have the same upper discrete mass dimension, given by a different formula from those in low dimensions: for every \(x\in A\), one has
$$
\dim_{\mathcal{UM}}^{\text{dis}}\big(\Delta^2(A)\big)=\dim_{\mathcal{UM}}^{\text{dis}}\big(\Delta^2(A,x)\big)=
\left\{
\begin{array}{ll}
     \frac{3\theta}{2}&\text{ when }0<\theta\leq\frac{1}{2} \\
     \frac{1+\theta}{2}&\text{ when }\frac{1}{2}\leq\theta\leq1
\end{array}
\right.
$$
First, we show the upper bound for $\dim_{\mathcal{UM}}^{\text{dis}}\big(\Delta^2(A)\big)$: our goal is to improve the estimate in \eqref{dgl}. Note that
$$
\|D_nu+v\|^2=\|v\|^2+2D_n\langle u,v\rangle+D_n^2\|u\|^2
$$
Since we know
$$
0\leq\|u\|^2\lesssim_d q^2\quad\quad\quad
\left|\langle u,v\rangle\right|\leq\sum_{j=1}^d|u_j||v_j|\lesssim_d qL_{n-1}
$$
it follows that
\begin{align*}
    \left|\Delta^2(A_{n})\cap[0,N]\right|&\lesssim_d\left|A_{n-1}-A_{n-1}\right|\cdot q^2\cdot qL_{n-1}\\
    &\lessapprox s_n^{o(1)}q^3\\
    &=s_n^{o(1)}(D_nq)^{3\theta}\left(\frac{q^{1-\theta}}{D_n^{\theta}}\right)^{3}\\
    &\leq s_n^{o(1)}N^{\frac{3\theta}{2}}\left(\frac{s_n^{1-\theta}}{D_n^{\theta}}\right)^{3}\\
    &\lessapprox s_n^{o(1)}N^{\frac{3\theta}{2}}=N^{\frac{3\theta}{2}+o(1)}
\end{align*}
Rewrite $\|D_nu+v\|^2$ as
$$
\|D_nu+v\|^2=\|v\|^2+D_nz\text{ where }z=2\langle u,v\rangle+D_n\|u\|^2
$$
Since we have
$$
|z|\lesssim_d qL_{n-1}+D_nq^2\lesssim D_nq^2
$$
it follows that
\begin{align*}
    \left|\Delta^2(A_{n})\cap[0,N]\right|&\lesssim_d\left|A_{n-1}-A_{n-1}\right|\cdot D_nq^2\\
    &\lessapprox s_{n}^{o(1)}D_nq^2\\
    &=s_{n}^{o(1)}(D_nq)^{1+\theta}\frac{q^{1-\theta}}{D_n^{\theta}}\\
    &\leq s_{n}^{o(1)} N^{\frac{1+\theta}{2}}\frac{s_n^{1-\theta}}{D_n^{\theta}}\\
    &\lessapprox s_{n}^{o(1)} N^{\frac{1+\theta}{2}}=N^{\frac{1+\theta}{2}+o(1)}
\end{align*}
Therefore, we obtain
\begin{align*}
    \left|\Delta^2(A_{n})\cap[0,N]\right|\lessapprox N^{\min\left\{\frac{3\theta}{2},\frac{1+\theta}{2}\right\}+o(1)}
\end{align*}
which implies
$$
\dim_{\mathcal{UM}}^{\text{dis}}\big(\Delta^2(A)\big)\leq
\left\{
\begin{array}{ll}
     \frac{3\theta}{2}&\text{ when }0<\theta\leq\frac{1}{2} \\
     \frac{1+\theta}{2}&\text{ when }\frac{1}{2}\leq\theta\leq1
\end{array}
\right.
$$
Finally, we prove the lower bound for $\dim_{\mathcal{UM}}^{\text{dis}}\big(\Delta^2(A,x)\big)$: fix $x\in A$. Then there exists $n_x\in\N$ such that $x\in A_{n_x}$. Since $A_{n_x}$ is a Cartesian product, we can choose $a\in A_{n_x}$ such that
$$
a-x=(v_1,0,\ldots,0)\in A_{n_x}-A_{n_x}\quad\quad\quad v_1\neq0
$$
For every $n>n_{x}$, we know
$$
a+D_n\{0,\ldots,s_n-1\}^d\subseteq A_n
$$
Let
$$
J=\{0,1,\ldots,s_n-1\}\quad\quad\quad W=\left\{u_2^2+\cdots+u_d^2:0\leq u_i\leq s_n-1\right\}
$$
For every $(j,w)\in J\times W$, we define
$$
F(j,w)=D_n^2(j^2+w)+2D_nv_1j+v_1^2\quad\quad\quad \mu_{F}(\lambda)=\left|\left\{(j,w)\in J\times W:F(j,w)=\lambda\right\}\right|
$$
Then
$$
\left|\Delta^2(A_{n},x)\right|\geq|F(J,W)|\geq\frac{|J||W|}{\|\mu_{F}\|_{\infty}}
$$
Now we show an upper bound for $\|\mu_{F}\|_{\infty}$: suppose that
$$
F(j_1,w_1)=F(j_2,w_2)
$$
Then
$$
D_n(j_1^2+w_1-j_2^2-w_2)=-2v_1(j_1-j_2)
$$
which implies
$$
h\mid(j_1-j_2)\text{ where }h=\frac{D_n}{\gcd(D_n,2v_1)}
$$
If $\lambda$ and $j$ are fixed, then $w$ is uniquely determined. In other words, we know
$$
\|\mu_{F}\|_{\infty}\leq\left|\left\{j_1\in J:j_1\equiv j_2\pmod{h}\right\}\right|\leq 1+\frac{s_n}{h}\leq1+\frac{2|v_1|s_n}{D_n}\lesssim_{|v_1|}1+\frac{s_n}{D_n}
$$
Since $|J|=s_n$ and $|W|\geq s_n^{2-o(1)}$, it follows that
$$
\left|\Delta^2(A_{n},x)\right|\gtrsim_{|v_1|}\frac{s_n^{3-o(1)}D_n}{s_n+D_n}\geq\frac{1}{2}s_n^{2-o(1)}\min\{s_n,D_n\}\gtrapprox\min\left\{s_n^3,s_n^{1+\frac{1}{\theta}}\right\}
$$
which implies
$$
\dim_{\mathcal{UM}}^{\text{dis}}\big(\Delta^2(A,x)\big)\geq
\left\{
\begin{array}{ll}
     \frac{3\theta}{2}&\text{ when }0<\theta\leq\frac{1}{2} \\
     \frac{1+\theta}{2}&\text{ when }\frac{1}{2}\leq\theta\leq1
\end{array}
\right.
$$
\end{remark}
\begin{remark}
For $\theta=1$, one can easily show that the construction in Example \ref{sharp} gives us the same conclusion with
$$
\dim_{\mathcal{H}}^{\text{dis}}(A)=\dim_{\mathcal{UM}}^{\text{dis}}(A)=\dim_{\mathcal{UB}}^{\text{dis}}(A)=d
$$
and
$$
\dim_{\mathcal{UM}}^{\text{dis}}\big(\Delta^2(A)\big)=\dim_{\mathcal{UM}}^{\text{dis}}\big(\Delta^2(A,x)\big)=
\left\{
\begin{array}{ll}
     \frac{1}{2}&\text{if }d=1 \\
     1&\text{if }d\geq2
\end{array}
\right.
$$
for all $x\in A$.
\end{remark}
\begin{eg}\label{peg}
This construction is inspired by \cite{bishop2017fractals} and illustrates the distinction between upper discrete mass dimension and upper discrete Banach dimension, as well as how this distinction affects the study of the unpinned and pinned distance problems. Moreover, it shows that discrete Hausdorff dimension alone cannot provide a nontrivial upper bound for the upper discrete mass dimension of the unpinned distance set.

Let
$$
B_n=Q_{r_n}+\left(R_n,0,\ldots,0\right)\quad\quad\quad A=\bigcup_{n\in\N}B_n
$$
where
$$
r_n=2^n\quad\quad\quad R_n=2^{2^n}
$$
Clearly, we know
$$
\dim_{\mathcal{H}}^{\text{dis}}(A)=\dim_{\mathcal{UM}}^{\text{dis}}(A)=0\quad\quad\quad\dim_{\mathcal{UB}}^{\text{dis}}(A)=d
$$
and
$$
\Delta^2(A)=\Delta^2(\Z^d)
$$
Hence, by Example \ref{eg1}, we have
$$
\mathcal{N}_{\theta}^{*}\big(\Delta^2(A)\big)>0
\left\{
\begin{array}{ll}
    \text{for all }\theta\in\left[0,\frac{1}{2}\right] &\text{if }d=1  \vspace{0.1cm}\\
     \text{for all }\theta\in[0,1) &\text{if }d=2 \vspace{0.1cm}\\
     \text{for all }\theta\in[0,1] &\text{if }d\geq3
\end{array}
\right.
$$
or equivalently,
$$
\dim_{\mathcal{UM}}^{\text{dis}}\big(\Delta^2(A)\big)=
\left\{
\begin{array}{ll}
   \frac{1}{2}  &\text{if }d=1  \vspace{0.1cm}\\
   1  &\text{if }d\geq2
\end{array}
\right.
$$
Fix $x\in A$. If $y\in B_n$ satisfies $\|x-y\|^2\leq N$, then
\begin{equation}
    \sqrt{N}\geq\|x-y\|\geq|x_1-y_1|\geq R_n-\frac{r_n}{2}-|x_1|\geq\frac{R_n}{2}-|x_1|
\end{equation}
which implies 
$$
n\leq M\text{ where }M=\left\lfloor\log_2\left(\log_2\left(2\left(\sqrt{N}+|x_1|\right)\right)\right)\right\rfloor
$$
Therefore, we obtain
\begin{align*}
    \left|\Delta^2(A,x)\cap[0,N]\right|&\leq\left|\left\{y\in A:\|x-y\|^2\leq N\right\}\right|\\
    &\leq\sum_{\frac{R_n}{2}\leq\sqrt{N}+|x_1|}|B_n|\\
    &\lesssim_d\sum_{\frac{R_n}{2}\leq\sqrt{N}+|x_1|}r_n^d\\
    &=\sum_{n=1}^{M}2^{nd}\\
    &\lesssim_d2^{Md}\\
    &\leq\left[\log_2\left(2\left(\sqrt{N}+|x_1|\right)\right)\right]^d\lesssim_x\left(1+\log_2N\right)^d
\end{align*}
which implies for every $\theta>0$, we get
$$
\mathcal{N}_{\theta}^{*}\big(\Delta^2(A,x)\big)=\limsup_{N\to\infty}\frac{\left|\Delta^2(A,x)\cap[0,N]\right|}{N^{\theta}}\leq\limsup_{N\to\infty}\frac{\left(1+\log_2N\right)^d}{N^{\theta}}=0
$$
In other words, we know
$$
\dim_{\mathcal{UM}}^{\text{dis}}\big(\Delta^2(A,x)\big)=0\text{ for all }x\in A
$$
\end{eg}
\begin{remark}
Example \ref{peg} shows that the upper discrete mass dimension cannot, in general, be replaced by the upper discrete Banach dimension in the pinned distance problem.
\end{remark}
\begin{eg}\label{manypin}
This construction is a modification of Example \ref{peg} and illustrates the distinction between discrete Hausdorff dimension and upper discrete mass dimension. Moreover, it shows that discrete Hausdorff dimension alone cannot provide a nontrivial upper bound for the upper discrete mass dimension of pinned distance sets, even when one is allowed to choose the pin.

Let $d\geq2$. Define
$$
R_n=2^{2^n}\quad\quad\quad t_n=\frac{R_n}{2^n}
$$
and
$$
B_n=\left\{(R_n+i,j,0,\ldots,0)\in\Z^d:0\leq i,j\leq t_n-1\right\}\quad\quad\quad A=\bigcup_{n\in\N}B_n
$$
We claim that
$$
\dim_{\mathcal{H}}^{\text{dis}}(A)=0\quad\quad\quad\dim_{\mathcal{UM}}^{\text{dis}}(A)=\dim_{\mathcal{UB}}^{\text{dis}}(A)=2
$$
and
$$
\dim_{\mathcal{UM}}^{\text{dis}}\big(\Delta^2(A,x)\big)=1\text{ for all }x\in A
$$
Let
$$
Q_n'=\left([R_n,R_n+t_n]\times[0,t_n]^{d-1}\right)\cap\Z^d
$$
Then
$$
B_n\subseteq Q_n'\quad\quad\quad\text{diam}(Q_n')\lesssim_d t_n\quad\quad\quad\max_{x\in Q_n'}\|x\|\geq R_n
$$
which implies for every $\alpha>0$, we know
$$
\mathscr{H}_{\alpha}(A)\leq\sum_{n=1}^{\infty}\left(\frac{\text{diam}(Q_n')}{\displaystyle\max_{x\in Q_n'}\|x\|}\right)^{\alpha}\lesssim_{d,\alpha}\sum_{n=1}^{\infty}\left(\frac{t_n}{R_n}\right)^{\alpha}<\infty
$$
Hence, we have $\dim_{\mathcal{H}}^{\text{dis}}(A)=0$. Recall that
$$
Q_{4R_n}=[-2R_n,2R_n]^d\cap\Z^d
$$
Then
$$
B_n\subseteq Q_{4R_n}\quad\quad\quad|A\cap Q_{4R_n}|\geq|B_n|=\frac{R_n^2}{2^{2n}}
$$
which implies for every $0\leq\theta<\frac{2}{d}$, we obtain
$$
\mathcal{N}_{\theta}^{*}(A)\geq\limsup_{n\to\infty}\frac{|A\cap Q_{4R_n}|}{|Q_{4R_n}|^{\theta}}=\infty
$$
Hence, we get
$$
2\leq\dim_{\mathcal{UM}}^{\text{dis}}(A)\leq\dim_{\mathcal{UB}}^{\text{dis}}(A)\leq\dim_{\mathcal{UB}}^{\text{dis}}\left(\Z^2\right)=2
$$
Fix $x\in A$. For every $y\in B_n$, we know
$$
\|x-y\|\leq\|x\|+\|y\|\lesssim_{x}R_n
$$
which implies there exists $C_x>0$ such that
$$
\Delta^2(B_n,x)\subseteq[0,C_xR_n^2]
$$
Since each squared distance determined by $x$ and points in $B_n$ occurs at most $\lesssim_{x,\epsilon}R_n^{2\epsilon}$ times, it follows that
$$
\left|\Delta^2(A,x)\cap[0,C_xR_n^2]\right|\geq\left|\Delta^2(B_n,x)\right|\gtrsim_{x,\epsilon}\frac{|B_n|}{R_n^{2\epsilon}}=\frac{R_n^{2-2\epsilon}}{2^{2n}}
$$
For every $0\leq\theta<1$ and every $0<\epsilon<1-\theta$, we have
$$
\mathcal{N}_{\theta}^{*}\big(\Delta^2(A,x)\big)\geq\limsup_{n\to\infty}\frac{\left|\Delta^2(A,x)\cap[0,C_xR_n^2]\right|}{(C_xR_n^2)^{\theta}}>0
$$
which implies
$$
1\leq\dim_{\mathcal{UM}}^{\text{dis}}\big(\Delta^2(A,x)\big)\leq\dim_{\mathcal{UM}}^{\text{dis}}\big(\Delta^2(A)\big)\leq\dim_{\mathcal{UM}}^{\text{dis}}\left(\Z\right)=1
$$
\end{eg}
\begin{remark}
For $d=1$, the analogous construction using discrete intervals
$$B_n=\{R_n,\ldots,R_n+t_n-1\}$$
gives
$$
\dim_{\mathcal{H}}^{\text{dis}}(A)=0\quad\quad\quad\dim_{\mathcal{UM}}^{\text{dis}}(A)=\dim_{\mathcal{UB}}^{\text{dis}}(A)=1
$$
and
$$
\dim_{\mathcal{UM}}^{\text{dis}}\big(\Delta^2(A)\big)=\dim_{\mathcal{UM}}^{\text{dis}}\big(\Delta^2(A,x)\big)=\frac{1}{2}\text{ for all }x\in A
$$
\end{remark}
\begin{eg}\label{egggg5555}
This construction is also a modification of Example \ref{peg} and shows that no nontrivial upper bound for the upper discrete mass dimension of an unpinned distance set can depend solely on the discrete Hausdorff dimension, upper discrete mass dimension, or upper discrete Banach dimension of the underlying set, even when all three dimensions vanish simultaneously.

Let $d\geq 2$. Enumerate $\mathbb Z_{\geq 0}^2=\{(a_n,b_n):n\in\mathbb N\}$ by setting
\[
(a_n,b_n)=(j,k-j)\quad\quad\quad n=\frac{k(k+1)}{2}+j+1
\]
for every $k\geq 0$ and $0\leq j\leq k$.
In particular, $\max\{a_n,b_n\}\leq n$. Let $R_n=2^{2^n}$ and define
\[
A=\bigcup_{n\in\N}
\left\{
(R_n,0,\ldots,0),
(R_n+a_n,b_n,0,\ldots,0)
\right\}
\]
Fix $t\in\Z^d$ and $N>4$. The projection of the translated cube $t+Q_N$ onto the first coordinate axis is
$$
I=\left[t_1-\frac{N}{2},t_1+\frac{N}{2}\right]
$$
Choose $M\in\N$ such that $R_M<N\leq R_{M+1}$. Since $a_n\leq n\leq R_n$ and $R_{n+1}=R_n^2$, the gaps between
consecutive intervals $I_n=[R_n,R_n+a_n]$ are
$$
R_{n+1}-(R_n+a_n)\geq R_n^2-2R_n\geq2R_n\geq2N\text{ for all }n\geq M+1
$$
Thus, we know
$$
\left|\left\{n\in\N:I\cap I_n\neq\varnothing\right\}\right|\leq M+1\lesssim\log_2{\log_2{N}}
$$
Since each interval $I_n$ corresponds to at most two points of $A$, it follows that
$$
|A\cap(t+Q_N)|\leq2\left|\left\{n\in\N:I\cap I_n\neq\varnothing\right\}\right|\lesssim\log_2{\log_2{N}}
$$
Consequently, we have
\[
\sup_{t\in\mathbb Z^d}|A\cap(t+Q_N)|\lesssim\log_2{\log_2{N}}
\]
which implies
\[
\dim_{\mathcal H}^{\mathrm{dis}}(A)
=\dim_{\mathcal{UM}}^{\mathrm{dis}}(A)
=\dim_{\mathcal{UB}}^{\mathrm{dis}}(A)=0
\]
On the other hand, by construction, we obtain
\[
\Delta^2\left(\Z^2\right)\subseteq\Delta^2(A)
\]
Hence, by Example \ref{eg1}, we get
\[
\dim_{\mathcal{UM}}^{\mathrm{dis}}(\Delta^2(A))=1
\]
Furthermore, by the elementary pinned upper bound in Remark \ref{upperboundpin}, we know
$$
\dim_{\mathcal{UM}}^{\text{dis}}\big(\Delta^2(A,x)\big)=0\text{ for all }x\in A
$$
\end{eg}
\begin{remark}
For $d=1$, the analogous construction $A=\displaystyle\bigcup_{n\in\N}\{R_{n},R_n+n\}$ gives
$$
\dim_{\mathcal{H}}^{\text{dis}}(A)=\dim_{\mathcal{UM}}^{\text{dis}}(A)=\dim_{\mathcal{UB}}^{\text{dis}}(A)=0
$$
and
$$
\dim_{\mathcal{UM}}^{\text{dis}}\big(\Delta^2(A)\big)=\frac{1}{2}\quad\quad\quad\dim_{\mathcal{UM}}^{\text{dis}}\big(\Delta^2(A,x)\big)=0\text{ for all }x\in A
$$
\end{remark}
\section{Discussion}\label{D}
In this section, we discuss the limitations of various theorems and techniques related to our discrete $\Z^d$-variant of the Erd\H{o}s-Falconer unpinned and pinned distance problems and explain why they fail to yield the desired results.
\subsection{$\ell^p$-Improving Inequalities for Discrete Spherical Averages}
The question regarding $\ell^p$-improving inequalities for discrete spherical averages, asked by Wright and Hughes \cite{AHL_2020__3__959_0}, is as follows:
\begin{qt}[Wright and Hughes]
For every $1<p<2$, what is the optimal exponent $\eta_p$ such that
$$
\|f\ast\sigma_{\lambda}\|_{\ell^{p',\infty}(\Z^d)}\lesssim_{d,p}\lambda^{-\eta_p}|S_{\lambda}|\|f\|_{\ell^p(\Z^d)}
$$
\end{qt}
The sharp result concerning this question was proved in \cite{han2026sharpellpimprovingestimatesfixedradius} and stated as Theorem \ref{lp}; however, our Theorem \ref{mainthm} is still far from our Conjecture \ref{conj}. In order to achieve our Conjecture \ref{conj}, Theorem \ref{lp} needs to hold for $p=\frac{d}{d-1}$ and all $d\geq3$, so it seems hopeless to prove our Conjecture \ref{conj} via $\ell^p$-improving inequalities for discrete spherical averages.

For $5\leq r<d$, although Han–Yang's $\ell^p$-improving inequalities provide an ambient-dimensional bound, this bound is weaker than the finite Erd\H{o}s distance bound and is therefore omitted from Theorem \ref{mainthm}. Moreover, neither Han–Yang's nor Hughes's $\ell^p$-improving inequalities directly yield the stronger intrinsic-dimensional threshold $T_{\ell^p}(r,\theta)$. The main obstacles are the following:
\begin{enumerate}[label=(\alph*)]
    \item By Remark \ref{affrk}, we need $\ell^p$-improving theorem for every fixed positive-definite integral quadratic form $Q$ with the same $r$-dimensional exponents; however, Han-Yang's $\ell^p$-improving theorem is only for the standard discrete spherical operator.
    \item Although Hughes's $\ell^p$-improving theorem is applicable to positive-definite integral quadratic forms, Hughes's $\ell^p$-improving theorem is restricted to regular values $\lambda\in\Gamma_{Q}$, rather than all $\lambda\in\N$. Moreover, Hughes's $p$-range is not as sharp as Han-Yang's.
\end{enumerate}
\subsection{Erd\H{o}s Distinct Distances Problem}
From the proof of the finite Erd\H{o}s distance bound, one can see that the Erd\H{o}s conjecture implies our Conjecture \ref{conj}. Here, we discuss two variants of the Erd\H{o}s distance problem:
\begin{enumerate}[label=(\alph*)]
    \item Solymosi and Vu \cite{solymosi2003distinct} proved a stronger result for homogeneous sets: For a homogeneous set $A\subseteq\R^d$ with $|A|=n$, one has $\left|\Delta^2(A)\right|\gtrsim n^{\frac{2}{d}-\frac{1}{d^2}}$. Nevertheless, the assumption of homogeneity is too strong to apply to our problem. The only applicable case is the trivial one, which has already been considered in Example \ref{eg1} via the sum of squares theorems without Solymosi and Vu's theorem for homogeneous sets.
    \item The finite field $\F_q^d$-variant of the Erd\H{o}s-Falconer conjecture proposed by Iosevich and Rudnev can imply our Conjecture \ref{conj} for $\theta=\frac{1}{2}$ when $d\geq2$ is even or $d,q\equiv3\pmod{4}$ by applying Conjecture \ref{finitefields} to $A_{N,t}$ and taking $q\to\infty$, so the progress towards Conjecture \ref{finitefields} is also relevant to our Conjecture \ref{conj} when $\theta\leq\frac{1}{2}$; nevertheless, the result via applying the current best bound of Conjecture \ref{finitefields} to our problem is weaker than the finite Erd\H{o}s distance bound, so we omit it.
\end{enumerate}
\subsection{Finite Pinned Distance Problem}
The current best result on the finite pinned distance problem for $\R^2$ was proved by Katz and Tardos \cite{MR2065258}:
\begin{thm}[Katz and Tardos]\label{pinned2}
Let $A\subseteq\R^2$ with $|A|=n$. Then there exists $x\in A$ such that
$$
\left|\Delta^2(A,x)\right|\gtrapprox n^{\alpha_2}
$$
where
$$
\alpha_2=\frac{48-14e}{55-16e}=0.864137\ldots
$$
\end{thm}
Note that Solymosi and Vu's recursion is also applicable to the finite pinned distance problem. Similarly, the current best result on the finite pinned distance problem for higher dimensions is applying Solymosi and Vu's recursion to Theorem \ref{pinned2} (by applying Solymosi and Vu's recursion (b) to Theorem \ref{pinned2} to get results for even dimensions $d\geq2$ and then applying Solymosi and Vu's recursion (a) to these results for even dimensions $d\geq2$ to get results for odd dimensions $d+1$):
\begin{thm}[The current best result on the finite pinned distance problem]
Let $d\geq3$ and $A\subseteq\R^d$ with $|A|=n$. Then there exists $x\in A$ such that
$$
\left|\Delta^2(A,x)\right|\gtrapprox n^{\alpha_d}
$$
where
$$
\alpha_d=
\left\{
\begin{array}{ll}
    \frac{2(d+1)}{d^2+2d-8+\frac{6}{\alpha_2}} &\text{if }d\text{ is even}  \\
    \frac{2\alpha_2d^2}{\alpha_2(d^3+d^2-9d+9)+6d-6} &\text{if }d\text{ is odd} 
\end{array}
\right.
$$
\end{thm}
\begin{remark}\label{fpdw}
The existing finite pinned distance results do not directly yield corresponding results in our discrete setting. Indeed, applying these estimates to \(A_N\) may produce different good pins at different scales, whereas our problem requires a single fixed pin \(x\in A\). A similar difficulty arises when using pinned distance results over finite fields.
\end{remark}
\begin{remark}
Our approach via $\ell^p$-improving inequalities can also derive a finite pinned distance result for $A_N$. Nevertheless, by Remark \ref{fpdw}, this result is not useful for our discrete $\Z^d$-variant setting.
\end{remark}
\subsection{A Weaker Notion of Homogeneity: $s$-Adaptability}
There is a weaker notion of homogeneity called $s$-adaptability in the literature \cite{iosevich2014theory,grafakos2015multilinear}, which makes a connection between the Erd\H{o}s conjecture and the Falconer conjecture and defines their version of discrete Hausdorff dimension of a family of 1-separated sets $A_N\subseteq\R^d$ with $|A_N|=N$ based on the notion of $s$-adaptability.

We recall some definitions and theorems from \cite{iosevich2014theory} here:
\begin{defn}
Let $d\geq2$. We say that $A\subseteq\R^d$ is \textbf{1-separated} if $\|a-a'\|\geq1$ for all distinct $a,a'\in A$.
\end{defn}
\begin{defn}
Let $A_N\subseteq\R^d$ be a family of 1-separated sets with $|A_N|=N$. Suppose that
$$
\left\{x\in\R^d:d\left(x,\frac{A_N}{\text{diam}(A_N)}\right)\leq\frac{1}{4\text{diam}(A_N)}\right\}\subseteq[0,1]^d
$$
We say that $A_N$ is $\bm{s}$\textbf{-adaptable} if for every $\beta<s$, one has
$$
\text{diam}(A_N)\lesssim N^{\frac{1}{s}}\text{ and }\frac{1}{N^2}\sum_{a\neq a'}\frac{1}{\|a-a'\|^{\beta}}\lesssim\frac{1}{\text{diam}(A_N)^{\beta}}
$$
\end{defn}
One can see that the definition of $s$-adaptability is based on the Riesz kernels in the continuous setting.
\begin{defn}[Iosevich, Rudnev, and Uriarte-Tuero]
\textbf{Their version of discrete Hausdorff dimension} of a family of 1-separated sets $A_N\subseteq\R^d$ with $|A_N|=N$ is defined by
\begin{align*}
\sup\{\beta>0:
    &\text{ for every }\epsilon>0,\text{ there exists a family of sets }B_N\subseteq A_N \\
    &\text{ such that }|B_N|\gtrsim_{\epsilon}\frac{|A_N|}{N^{\epsilon}}\text{ and }B_N\text{ is }\beta\text{-adaptable}\}
\end{align*}
\end{defn}
The definition of their version of discrete Hausdorff dimension is based on the notion of $s$-adaptability.
\begin{thm}[Iosevich, Rudnev, and Uriarte-Tuero]\label{their}
Suppose that the Falconer distance conjecture holds to the extent that if a Borel probability measure $\mu$ supported on $E\subseteq[0,1]^d$ satisfies that $I_{\alpha}(\mu)\leq C_0<\infty$, where
$$
I_{\alpha}(\mu)=\iint\frac{1}{\|x-y\|^{\alpha}}d\mu(x)d\mu(y)
$$
for some $\alpha>s_0\geq\frac{d}{2}$, then $\mathcal{L}_1\big(\Delta(E)\big)\geq C(\alpha,C_0)>0$.

Let $A_N\subseteq\R^d$ be a family of sets with $|A_N|=N$ and with their version of discrete Hausdorff dimension $>s_0$. Then
$$
\left|\Delta(A_N)\right|\gtrapprox N^{\frac{1}{s_0}}
$$
In particular, if the Falconer conjecture is true, then the Erd\H{o}s conjecture is true for any family of sets $A_N\subseteq\R^d$ with their version of discrete Hausdorff dimension $>\frac{d}{2}$.
\end{thm}
\begin{remark}
Frostman's theory and energy integrals play a crucial role in the Falconer distance problem. There is also a discrete analog of Frostman's theory and energy integrals; see for example \cite[Chapter~8]{bishop2017fractals}. However, we cannot emulate the proof in the continuous setting because the Riesz kernels in the discrete setting are not symmetric and therefore disable the push-forward technique used in the continuous setting.
\end{remark}

The relation between their version of discrete Hausdorff dimension of a family of 1-separated sets and the original discrete Hausdorff dimension used in this paper for subsets of $\Z^d$ remains unclear; however, the two different discrete Hausdorff dimensions are equal in several examples, for instance, Delone sets in \cite{iosevich2014theory}.

Theorem \ref{their} depends on the progress towards the Falconer conjecture. By applying Theorem \ref{their} to $A_{N,t}$ with the current best thresholds for the Falconer conjecture, which are $\frac{5}{4}$ when $d=2$ and $\frac{d}{2}+\frac{1}{4}-\frac{1}{8d+4}$ when $d\geq3$, one can get a conditional bound which will partially overlap with Theorem \ref{mainthm2} only if their version of discrete Hausdorff dimension is greater than or equal to the original discrete Hausdorff dimension used in this paper; nonetheless, even if their version of discrete Hausdorff dimension is greater than or equal to the original discrete Hausdorff dimension used in this paper, Theorem \ref{mainthm2} is still stronger than the conditional bound when $\theta\geq\theta_{\text{stronger}}^{\mathrm{conditional}}$ or equivalently $\dim_{\mathcal{H}}^{\text{dis}}(A)\geq d_{\text{stronger}}^{\mathrm{conditional}}$, where
$$
\theta_{\text{stronger}}^{\mathrm{conditional}}=
\left\{
\begin{array}{ll}
    \text{Theorem \ref{mainthm2} is always stronger, so there is no such }\theta_{\text{stronger}}^{\mathrm{conditional}}&\text{if }2\leq d\leq 6\\
    \frac{(d-2)(2d+1)}{2d^2-d-2} &\text{if }d\geq7
\end{array}
\right.
$$
and
$$
d_{\text{stronger}}^{\mathrm{conditional}}=
\left\{
\begin{array}{ll}
    \text{Theorem \ref{mainthm2} is always stronger, so there is no such }d_{\text{stronger}}^{\mathrm{conditional}}&\text{if }2\leq d\leq 6\\
    \frac{2d(d+1)(d-2)}{2d^2-d-2} &\text{if }d\geq7
\end{array}
\right.
$$
Additionally, the conditional bound is inapplicable when $\theta\geq\theta_{\text{inapplicable}}^{\mathrm{conditional}}$, where
$$
\theta_{\text{inapplicable}}^{\mathrm{conditional}}=\left\{
\begin{array}{ll}
    \frac{4}{5} &\text{if }d=2  \\
    \frac{2d+1}{2(d+1)} &\text{if }d\geq3 
\end{array}
\right.
$$
but in the meantime, Theorem \ref{mainthm2} is applicable when $\theta\in[0,1)$ and hence can fill this gap. We also note that the conditional bound only holds for $\theta\geq\frac{1}{2}$ if their version of discrete Hausdorff dimension is greater than or equal to the original discrete Hausdorff dimension used in this paper.
\pgfmathsetmacro{\figSevenThetaCross}{75/89}
\pgfmathsetmacro{\figSevenDimCross}{560/89}

\pgfmathsetmacro{\figSevenThetaInapp}{15/16}

\pgfmathsetmacro{\figSevenThetaSwitch}{5/6}
\pgfmathsetmacro{\figSevenDimSwitch}{25/4}


\pgfplotsset{
    figSevenStyle/.style={
        width=12cm,
        height=8cm,
        axis lines=left,
        tick label style={font=\small},
        label style={font=\small},
        legend style={
            at={(0.05,0.95)},
            anchor=north west,
            draw=black,
            fill=white,
            fill opacity=1,
            text opacity=1,
            font=\small,
            legend columns=1
        },
        legend cell align=left
    }
}
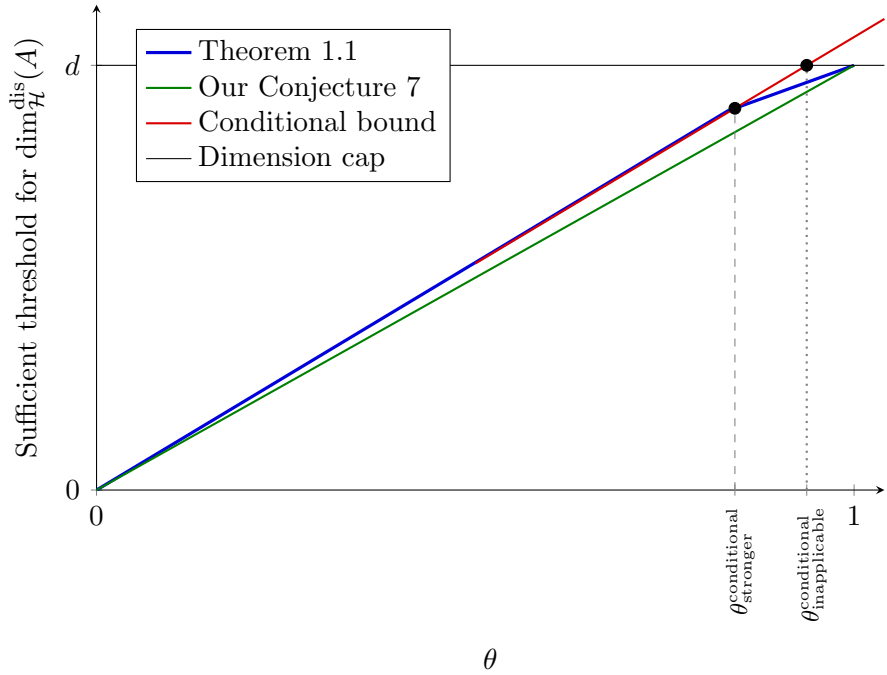
\begin{figure}[H]
\centering

\begin{tikzpicture}

\begin{axis}[
    figSevenStyle,
    xmin=0, xmax=1.04,
    ymin=0, ymax=8,
    xlabel={$\theta$},
    ylabel={Sufficient threshold for
        $\dim_{\mathcal H}^{\mathrm{dis}}(A)$},
    xtick={0,1},
    xticklabels={$0$,$1$},
    ytick={0,7},
    yticklabels={$0$,$d$},
    extra x ticks={
        \figSevenThetaCross,
        \figSevenThetaInapp
    },
    extra x tick labels={
        $\theta_{\mathrm{stronger}}^{\mathrm{conditional}}$,
        $\theta_{\mathrm{inapplicable}}^{\mathrm{conditional}}$
    },
    extra x tick style={
        tick label style={
            rotate=90,
            anchor=east,
            font=\scriptsize,
            yshift=-2pt
        }
    }
]

%
%

\addplot[
    blue!85!black,
    very thick,
    no marks
] coordinates {
    (0,0)
    (\figSevenThetaSwitch,\figSevenDimSwitch)
    (1,7)
};
\addlegendentry{Theorem \ref{mainthm2}}

\addplot[
    green!50!black,
    thick,
    domain=0:1,
    samples=2
] {7*x};
\addlegendentry{Our Conjecture \ref{conj}}

\addplot[
    red!85!black,
    thick,
    domain=0.5:1.04,
    samples=2
] {112*x/15};
\addlegendentry{Conditional bound}

\addplot[
    black,
    thin,
    domain=0:1.04,
    samples=2
] {7};
\addlegendentry{Dimension cap}

\addplot[
    gray,
    dashed,
    forget plot
] coordinates {
    (\figSevenThetaCross,0)
    (\figSevenThetaCross,\figSevenDimCross)
};

\addplot[
    gray,
    dotted,
    thick,
    forget plot
] coordinates {
    (\figSevenThetaInapp,0)
    (\figSevenThetaInapp,7)
};

\addplot[
    only marks,
    mark=*,
    mark size=2.2pt,
    black,
    forget plot
] coordinates {
    (\figSevenThetaCross,\figSevenDimCross)
    (\figSevenThetaInapp,7)
};

\end{axis}
\end{tikzpicture}

\caption{
Comparison of Theorem \ref{mainthm2}, our Conjecture \ref{conj},
and the conditional bound,
together with the dimension cap,
for $d=7$.
}
\label{fig:comparison-adapt-theta}

\end{figure}


\begin{figure}[H]
\centering

\begin{tikzpicture}

\begin{axis}[
    figSevenStyle,
    xmin=0, xmax=7.2,
    ymin=0, ymax=1.07,
    xlabel={
        $\dim_{\mathcal H}^{\mathrm{dis}}(A)$
    },
    ylabel={
        Lower bounds for
        $\dim_{\mathcal{UM}}^{\mathrm{dis}}
        (\Delta^2(A))$
    },
    xtick={0,7},
    xticklabels={$0$,$d$},
    ytick={0,1},
    yticklabels={$0$,$1$},
    extra x ticks={\figSevenDimCross},
    extra x tick labels={
        $d_{\mathrm{stronger}}^{\mathrm{conditional}}$
    },
    extra x tick style={
        tick label style={
            rotate=90,
            anchor=east,
            font=\scriptsize,
            yshift=-2pt
        }
    }
]

%
%

\addplot[
    blue!85!black,
    very thick,
    no marks
] coordinates {
    (0,0)
    (\figSevenDimSwitch,\figSevenThetaSwitch)
    (7,1)
};
\addlegendentry{Theorem \ref{mainthm2}}

\addplot[
    green!50!black,
    thick,
    domain=0:7,
    samples=2
] {x/7};
\addlegendentry{Our Conjecture \ref{conj}}

\addplot[
    red!85!black,
    thick,
    domain=3.733333333:7,
    samples=2
] {15*x/112};
\addlegendentry{Conditional bound}

\addplot[
    black,
    thin,
    domain=0:7,
    samples=2
] {1};
\addlegendentry{Dimension cap}

\addplot[
    black,
    thin,
    forget plot
] coordinates {
    (7,0)
    (7,1)
};

\addplot[
    gray,
    dashed,
    forget plot
] coordinates {
    (\figSevenDimCross,0)
    (\figSevenDimCross,\figSevenThetaCross)
};

\addplot[
    only marks,
    mark=*,
    mark size=2.2pt,
    black,
    forget plot
] coordinates {
    (\figSevenDimCross,\figSevenThetaCross)
};

\end{axis}
\end{tikzpicture}

\caption{
Comparison of the lower bounds for
$\dim_{\mathcal{UM}}^{\mathrm{dis}}
(\Delta^2(A))$
obtained from Theorem \ref{mainthm2},
our Conjecture \ref{conj}, and the conditional bound, for $d=7$.
}
\label{fig:comparison-adapt-dimension}

\end{figure}
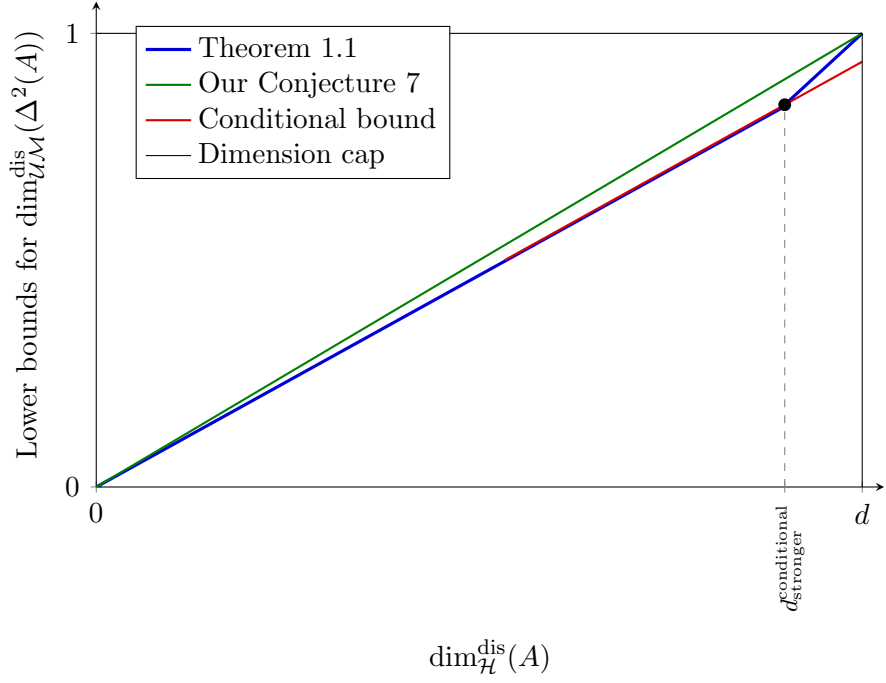
\section{Discrete Hausdorff Dimension of Example \ref{sharp}}\label{A}
In this section, we prove the lower bound $\dim_{\mathcal{H}}^{\text{dis}}(A)\geq d\theta$ for the construction in Example \ref{sharp} for all $d\in\N$. Our proof is based on the discrete Frostman theory developed in \cite[Chapter~8]{bishop2017fractals}. We establish a uniform energy bound for a sequence of probability measures supported on successive generations of the construction and then apply the characterization of discrete Hausdorff dimension in terms of asymptotic capacity.

Let $0<\theta\leq1$ and $0<\beta<d\theta$. Recall that
$$
A_{n}=A_{n-1}+D_{n}I_n
$$
where
$$
D_{n}=\ell_n+L_{n-1}\quad\quad\quad I_n=\{0,1,\ldots,s_n-1\}^d
$$
Define
$$
A'_n=A_n\setminus A_{n-1}\quad\quad\quad\mu_n=\frac{1}{|A_n'|}\sum_{x\in A_n'}\bm{1}_{\{x\}}
$$
In other words, $\mu_n$ is the probability measure on $A$ that gives all points in $A_n'$ equal mass. Write $x=a+D_nu\in A_n'$ and $y=b+D_nv\in A_n'$, where $a,b\in A_{n-1}$ and $u,v\in I_n\setminus\{\mathbf{0}\}$. For every $t\in\N$, we define
$$
B_t(v)=\left|\left\{z\in\Z^d:\|z-v\|_{\infty}\leq t\right\}\right|=(2t+1)^d
$$
Then
\begin{align*}
    \sum_{\substack{x\in A_n'\\\|u-v\|_{\infty}\leq1}}\frac{1}{1+\|x-y\|^{\beta}}\leq\sum_{a\in A_{n-1}}\sum_{\substack{u\in I_n\setminus\{\mathbf{0}\}\\\|u-v\|_{\infty}\leq1}}1\leq |A_{n-1}|\cdot B_1(v)=3^d|A_{n-1}|
\end{align*}
and for every integer $t\geq2$, we know
\begin{align*}
    \left|\left\{u\in I_n\setminus\{\mathbf{0}\}:\|u-v\|_{\infty}=t\right\}\right|\leq B_{t}(v)-B_{t-1}(v)=(2t+1)^d-(2t-1)^d\lesssim_d t^{d-1}
\end{align*}
Moreover, if $\|u-v\|_{\infty}\geq2$, then
$$
\|x-y\|\geq D_n\|u-v\|_{\infty}-L_{n-1}\geq\frac{D_n\|u-v\|_{\infty}}{2}
$$
which implies
\begin{align*}
    \sum_{\substack{x\in A_n'\\\|u-v\|_{\infty}\geq2}}\frac{1}{1+\|x-y\|^{\beta}}&\leq\sum_{a\in A_{n-1}}\sum_{\substack{u\in I_n\setminus\{\mathbf{0}\}\\\|u-v\|_{\infty}\geq2}}\frac{2^{\beta}}{D_n^{\beta}\|u-v\|_{\infty}^{\beta}}\\
    &\leq\frac{2^{\beta}|A_{n-1}|}{D_n^{\beta}}\sum_{t=2}^{s_n-1}\frac{\left|\left\{u\in I_n\setminus\{\mathbf{0}\}:\|u-v\|_{\infty}=t\right\}\right|}{t^{\beta}}\\
    &\lesssim_d \frac{2^{\beta}|A_{n-1}|}{D_n^{\beta}}\sum_{t=2}^{s_n-1} t^{d-1-\beta}\lesssim_{d,\beta}\frac{|A_{n-1}|s_n^{d-\beta}}{D_n^{\beta}}
\end{align*}
Therefore, we have
\begin{align*}
    \int_{A} F_{\beta}(x,y)d\mu_n(x)&=\frac{\|y\|^{\beta}}{|A_n'|}\sum_{x\in A_n'}\frac{1}{1+\|x-y\|^{\beta}}\\
    &\leq\frac{(\sqrt{d}L_n)^{\beta}}{(s_n^d-1)|A_{n-1}|}\left(\sum_{\substack{x\in A_n'\\\|u-v\|_{\infty}\leq1}}\frac{1}{1+\|x-y\|^{\beta}}+\sum_{\substack{x\in A_n'\\\|u-v\|_{\infty}\geq2}}\frac{1}{1+\|x-y\|^{\beta}}\right)\\
    &\lesssim_{d,\beta}\frac{L_{n}^{\beta}}{s_n^d}\left(1+\frac{s_n^{d-\beta}}{D_n^{\beta}}\right)\\
    &=\frac{L_{n}^{\beta}}{s_n^d}+\left(\frac{L_n}{s_nD_n}\right)^{\beta}\leq s_{n}^{\frac{\beta}{\theta}-d+o(1)}+1^{\beta}
\end{align*}
which implies
$$
\sup_{n\geq2}\mathcal{E}_{\beta}(\mu_n)=\sup_{n\geq2}\int_{A}\int_{A} F_{\beta}(x,y)d\mu_n(x)d\mu_n(y)<\infty
$$
For every finite subset $F\subseteq A$, there exists $n_0\geq2$ such that $A_{n}'\subseteq A\setminus F$ for all $n\geq n_0$. Thus, we obtain
$$
\mathrm{Cap}_{\beta}\left(A\setminus F\right)\geq\mathrm{Cap}_{\beta}(A_n')\geq\frac{1}{\mathcal{E}_{\beta}(\mu_n)}\geq\frac{1}{\displaystyle\sup_{n\geq2}\mathcal{E}_{\beta}(\mu_n)}>0
$$
which implies
$$
\mathrm{Cap}_{\beta}^{\infty}(A)=\inf_{\substack{F\subseteq A\\F\text{ finite}}}\mathrm{Cap}_{\beta}\left(A\setminus F\right)\geq\frac{1}{\displaystyle\sup_{n\geq2}\mathcal{E}_{\beta}(\mu_n)}>0
$$
By Theorem 8.1.3 in \cite{bishop2017fractals}, we get
$$
\dim_{\mathcal{H}}^{\text{dis}}(A)\geq\beta\text{ for all }0<\beta<d\theta
$$
which implies $\dim_{\mathcal{H}}^{\text{dis}}(A)\geq d\theta$.
\section{Acknowledgments}
The authors would like to thank HaoYun Yao for helpful comments on an early draft of this paper and for providing relevant references, and Chun-Yen Shen for suggesting an interesting research problem during NCTS One Day Workshop on Harmonic Analysis.
\clearpage
\phantomsection
\addcontentsline{toc}{section}{References}
\printbibliography
\end{document}